\documentclass[10pt]{amsart}

\usepackage{amssymb,amsthm,amsmath}
\usepackage[numbers,sort&compress]{natbib}
\usepackage{color}
\usepackage{graphicx}
\usepackage{tikz}
\usepackage{amsfonts}
\title[Almost periodic solution for NLS]{On the existence of full dimensional KAM torus for nonlinear Schr\"odinger equation }

\thanks{The first author is supported by the National Natural Science Foundation of China (No. 11671066). The third author is supported by
China Postdoctoral Science Foundation Grant (No. 2018M641050).}
\author{Hongzi Cong}
\address[Hongzi Cong]{School of Mathematical Sciences,
Dalian University of Technology, Dalian, Liaoning 116024, China} \email{conghongzi@dlut.edu.cn}

\author{Lufang Mi}
\address[Lufang Mi]{College of Science,  The Institute of Aeronautical Engineering and Technology,
Binzhou University,
Binzhou 256600,
 China} \email{milufang@126.com}

\author{Yunfeng Shi}
\address[Yunfeng Shi]{School of Mathematical Sciences,
Peking University,
Beijing 100871,
P. R. China} \email{yunfengshi18@gmail.com}

\author{Yuan Wu}
\address[Yuan Wu]{School of Mathematical Sciences,
Fudan University,
Shanghai 200433,
P. R. China} \email{14110840003@fudan.edu.cn}


\theoremstyle{plain}
\newtheorem{thm}{Theorem}[section]
 
\newtheorem{lem}[thm]{Lemma}
 
 \theoremstyle{definition}
\newtheorem{defn}[thm]{Definition}
 \theoremstyle{remark}
 \newtheorem{rem}[thm]{Remark}
 
 \numberwithin{equation}{section}
\begin{document}

% \newcommand{\N}{\mathbb{N}}
%%% ----------------------------------------------------------------------

\begin{abstract}
In this paper, we study the  following nonlinear Schr\"odinger equation
\begin{eqnarray}\label{maineq0}
\textbf{i}u_{t}-u_{xx}+V*u+\epsilon f(x)|u|^4u=0,\ x\in\mathbb{T}=\mathbb{R}/2\pi\mathbb{Z},
\end{eqnarray}
where $V*$ is the Fourier multiplier defined by $\widehat{(V* u})_n=V_{n}\widehat{u}_n, V_n\in[-1,1]$ and $f(x)$ is Gevrey smooth.
It is shown that for $0\leq|\epsilon|\ll1$, there is some $(V_n)_{n\in\mathbb{Z}}$ such that,  the equation (\ref{maineq0}) admits a time almost periodic solution (i.e., full dimensional KAM torus) in the Gevrey space. This extends results  of Bourgain \cite{BJFA2005}  and Cong-Liu-Shi-Yuan \cite{CLSY} to the case that the nonlinear perturbation depends explicitly on the space variable $x$. The main difficulty here is the absence of zero momentum of the equation.
\end{abstract}

%%% ----------------------------------------------------------------------
\maketitle
%%% ----------------------------------------------------------------------
\section{Introduction and main result}
%Consider a nearly-integrable Hamiltonian systems of $n$-freedom
%\begin{equation}\label{yuan1}H=H_0(I)+\epsilon H_{1}(\theta,I),\end{equation}
%with the standard symplectic  structure $\mathrm{d}\theta\wedge\mathrm{d}I$ on $\mathbb{T}^n\times\mathbb{R}^{n}$ and the angle-action variable $(\theta,I)$ belongs to some domain
%$\mathbb{T}^n\times D\subseteq\mathbb{T}^n\times\mathbb{R}^{n}$. Assume the unperturbed Hamiltonian  $H_0(I)$ is independent of $\theta$ and satisfies Kolmogorov non-degenerate condition
%$$\mathrm{det}(\partial^2H_0(I))\neq0,\ I\in D.$$
%Also assume $H_0,H_1$ are smooth sufficiently. Then the well-known Kolmogorov-Arnold-Moser (KAM) theorem (\cite{A1961,K1954,M1962}) claims that any invariant tori of the unperturbed $H_0$ with prescribed Diophantine frequency $\omega(I_0)=\frac{\partial{H_0}(I_0)}{\partial I}$ for some $I_0\in D$ persist under a small perturbation $\epsilon H_1(I,\theta)$. This theorem is now called the classical KAM theorem and the persisted tori called full dimensional
%KAM tori.
In this paper, we focus on  the nonlinear Schr\"odinger equation (NLS) with periodic boundary conditions
\begin{equation}\label{maineq}
\sqrt{-1}u_{t}- u_{xx}+V*u+\epsilon f(x)|u|^4u=0,\  x\in\mathbb{T},
\end{equation}
where $\textbf{i}=\sqrt{-1}$, $V*$ is a Fourier multiplier defined by
\begin{equation}
V*u=\sum_{n\in\mathbb{Z}}V_n\widehat{u}_ne^{\textbf{i}nx},\  V_n\in[-1,1],
\end{equation} $f(x)$ is $2\pi$-periodic  and real analytic in $x,y$.
Written in Fourier modes $ (q_{n})_{n\in \mathbb{Z}}$,\ then (\ref{maineq}) can be rewritten as
\begin{eqnarray*}
\label{002} \dot{q}_{n}= \mathbf{i}\frac{\partial H}{\partial \overline{q}_{n}}
\end{eqnarray*}
with the Hamiltonian
\begin{equation}
\label{003}
H(q,\overline{q})=\sum_{n\in\mathbb{Z}}(n^{2}+V_{n})|q_{n}|^{2}
+\epsilon\sum_{n\in\mathbb{Z}}\underset{n_{1}-n_{2}+n_{3}-n_{4}+n_{5}-n_{6}=-n}{\sum}
\widehat{f}(n)q_{n_{1}}\overline{q}_{n_{2}}q_{n_{3}}\overline{q}_{n_{4}}q_{n_{5}}\overline{q}_{n_{6}}.
\end{equation}
Our aim is to show the existence of  almost periodic solutions for
such a family of NLS.
%We consider the families of the nonlinear Schr\"odinger equations on the circle with of the form
%\begin{eqnarray}\label{0124001}
%\textbf{i}u_{t}-u_{xx}+V*u+\epsilon g(x,|u|^2)u=0,\qquad x\in\mathbb{T}=\mathbb{R}/\mathbb{Z},
%\end{eqnarray}
%and is real analytic in $y$ in a neighborhood of $y=0$.

In the last few decades, the persistence of the invariant tori for NLS has been drawn a lot of attentions by many authors.  To  this end, one considers the infinite dimensional Hamiltonian of the form
$$H=N+\epsilon P(\theta,I,z,\bar{z}), $$
with the symplectic structure $\mathrm{d} \theta\wedge\mathrm{d} I+\sqrt{-1}\mathrm{d}z\wedge\mathrm{d}\bar{z}$
 on $\mathbb{T}^d\times\mathbb{R}^{d}\times\mathcal{H}\times\mathcal{H}\ni(\theta,I,z,\bar{z})$ and
$$N=\sum_{i=1}^{d}\omega_iI_i+\sum_{j\geq1}\Omega_{j}z_j\bar{z}_{j},\qquad  1\leq d<\infty,$$
where $\omega=(\omega_1,\omega_2,\cdots,\omega_d)$ is called tangent frequency vector, $(\Omega_j)_{j\geq1}$ is called the normal frequency vector, and $P=P(\theta,I,z,\bar{z})$ is a perturbation. The unperturbed Hamiltonian $N$ has a special invariant torus
$$\mathcal{T}_0=\mathbb{T}^d\times\{I=0\}\times\{z=0\}\times\{\bar{z}=0\},$$
and all solutions starting on $\mathcal{T}_0$ are quasi-periodic with the frequency $\omega$.
Under suitable assumptions on $N$ and $P$, it can be proved that for ``most'' frequency $\omega$, the tori $\mathcal{T}_0$ can be persisted for some small perturbation $\epsilon P$ (see \cite{K1987,KB2000,W1990} for example).
 However, the KAM theorem of this type depends heavily on the fact that the spatial dimension of the PDEs equals to $1$. Bourgain \cite{BA1998,BB2005} developed a new method initiated by Craig-Wayne \cite{CW1993} to deal with the KAM tori for the PDEs in high spatial dimension, based on the Newton iteration, Fr\"{o}hlich-Spencer techniques, Harmonic analysis and semi-algebraic set theory (see \cite{BB2005}). This method is now called C-W-B method. We also mention Eliasson-Kuksin \cite{EK2010} where the KAM theorem is extended  to deal with higher spatial dimensional nonlinear Schr\"{o}dinger equation. In addition, the classical KAM theory is also developed to deal  some 1D PDEs of unbounded perturbation. See, for example, \cite{KB2000,KPB2003,LY2011,ZGY2011,BBM2014,BBM2016,FP2015} for the details. In the all above works, the obtained KAM tori are lower (finite) dimension. %considering that the Hamiltonian PDEs are infinite dimensional.
Naturally, the following problem is interesting:
 \vskip8pt
 {\it Can the full dimensional invariant tori be expected with a suitable decay, for example, $I_n\sim|n|^{-S}$ with some $ S>0$ as $|n|\rightarrow+\infty$ ? }
\vskip8pt
  The existence of the  full dimensional KAM tori  with polynomial decay rate $I_n\sim|n|^{-S}$ is still open up to now. See \cite{KZ2004} for the details.
   The first result about the existence of the full dimensional tori (or almost periodic solutions) for Hamiltonian PDEs was obtained by Bourgain \cite{B1996}.
   Precisely, using C-W-B method the almost periodic solutions (in time) of the form
   \begin{equation}
   u(t,x)=\sum_{n\geq 1}a_n\cos  \omega_n t\ \phi_n(x)
   \end{equation}
   were constructed for 1D nonlinear wave equation (NLW)
   \begin{equation}
   u_{tt}-u_{xx}+V(x)u+\epsilon f(u)=0
   \end{equation} under Dirichlet boundary conditions, where $\omega_n\approx \sqrt{\lambda_n}$ and $\lambda_n$ is the Dirichlet spectrum of $-\partial_{xx}+V(x)$.
   Here, a strong decay assumption $|a_n|\rightarrow 0$ is needed for the amplitude $a_n$.
    P\"oschel  \cite{P2002} proved the existence of almost periodic solutions for NLS equation by the KAM method (also see {\cite{G2012},\cite{GX2013},\cite{NG2007},\cite{WG2011}}). The basic idea in these papers is to use repeatedly (infinitely many times) the KAM theorem dealing with lower dimensional KAM tori. That is why the amplitude (or action) of those almost periodic solutions decay extremely fast. In fact, the decay rate is defined implicitly and much more fast than $a_n\sim e^{-|n|^{C}}, C>1$.  See more comments in \cite{BZ2004}. Recently, the invariant tori of full dimensions for second KdV equations with the external parameters were constructed by Geng-Hong \cite{GH2017}, where noting that the nonlinear term contains the derivatives.

 Another way is due to Bourgain in \cite{BJFA2005}  where 1D NLS with periodic boundary condition
 was investigated (see also \cite{P1990} by P\"{o}schel where infinite dimensional Hamiltonian systems with short range was considered). It was shown in \cite{BJFA2005} that 1D NLS has a full dimensional KAM torus of prescribed frequencies  where the actions of the torus obey the estimates \begin{equation}\label{083101}\frac{1}{2}e^{-r{|n|^{\theta}}}\leq I_n\leq 2e^{-r{|n|^{\theta}}}
\end{equation}
with $\theta=\frac{1}{2}$.
This is up to now only existence result about the full dimensional KAM tori with a slower decay rate than $I_n\sim e^{-|n|^{S}}, S>1$. In a different way, Bourgain constructed the full dimensional tori directly, where a more complicated small divisor problem has to be dealt with. An important observation by Bourgain is the following:
Let $(n_i)$ be a finite set of modes satisfy $|n_1|\geq |n_2|\geq \cdots$ and
\begin{equation}\label{M}
n_1-n_2+n_3-\cdots=0.
\end{equation}
In the case of a `near' resonance, there is also a relation
\begin{equation}\label{M'}
n^2_1
-n^2_2
+n^2_3
-\cdots=o(1).
\end{equation}
Unless $n_1=n_2$, one may then control $|n_1|+|n_2|$ from (\ref{M}), (\ref{M'}) by
$\sum_{j\geq 3}
|n_j|.$ More recently, Cong-Liu-Shi-Yuan \cite{CLSY} extended Bourgain's results to the any $\theta\in (0,1)$.

\iffalse %%%%%%%%
%%%%%%%%%
%%%%%%%%%%%%
 In order to fulfill the decay rate, Bourgain introduced a weight function $\sum_n(2a_n+k_n+k_n'){|n|^{1/2}}$ for a polynomial Hamiltonian
\[\sum_{a,k,k^\prime} B_{a,k,k^\prime} \prod_{n} I_n^{a_n} q_n^{k_n}{\bar q}_n^{k^\prime_n}. \]
In Remark 2 (p. 67, \cite{Bour2005JFA}), Bourgain stated that

\vskip8pt

{\it  the weight function $\sum_n(2a_n+k_n+k_n'){|n|^{1/2}}$ may have been replaced by expression $\sum_n(2a_n+k_n+k_n)|n|^{\theta}$ for some $0<\theta<1$}.

\vskip8pt

In the present paper, one of our main aims is to prove the Bourgain's statement. Thus, it will be shown that there is  a full dimensional KAM tori of prescribed frequencies for the NLS with decay rate of \begin{equation}\label{083102}\frac{1}{2}e^{-r{|n|^{\theta}}}\leq I_n\leq 2e^{-r{|n|^{\theta}}}, n\in\mathbb{Z},\ r>0, \ 0<\theta<1.
\end{equation}

\fi%%%%

Note that the condition (\ref{M}) is no longer valid for the Hamiltonian (\ref{003}). But if the function $f(x)$ is Gevrey smooth with $\mu>0$, then one has
\begin{align}\label{M''}
|\widehat{f}(n)|\leq Ce^{-\mu|n|^\theta},\ \mu>0,\ \theta\in (0,1).
\end{align}
Thus we use the property (\ref{M''}) to guarantee $|n_1|+|n_2|$ can be controlled by $\sum_{j\geq 3}|n_j|+|n|$.

To state our result precisely, we will give some definitions firstly.
\begin{defn}\label{007}
 Given $ 0 < \theta < 1$ and $ r > 0 $,\ we define the Banach space $ \mathfrak{H}_{r,\infty}$  consisting of all complex sequences $ q = (q_{n})_{n\in\mathbb{Z}}$ with
\begin{eqnarray}
\label{H0} \|q \|_{r,\infty}= \sup_{n\in \mathbb{Z}}|q_{n}|e^{r|n|^{\theta}} < \infty.
\end{eqnarray}
\end{defn}
\begin{defn}
Denote $ \|x\| = dist (x,\mathbb{Z})$. A vector $\omega=(\omega_n)_{n\in\mathbb{Z}}$ is called to be Diophantine if there exists a real number $ \gamma > 0 $ such that the following resonance issues
\begin{eqnarray}
\label{005} \left\| \sum_{n\in \mathbb{Z}}l_{n}\omega_{n}\right\|\geq \gamma \prod_{n\in \mathbb{Z}}\frac{1}{1+l^{2}_{n}|n|^{4}}
\end{eqnarray}
hold, where $ 0 \neq l = (l_{n})_{n\in \mathbb{Z}}$ is a finitely supported sequence of integers and
\[
|n|= \max\{ 1,n,-n \}.
\]
\end{defn}
\begin{thm}\label{Thm}
Given $r>0$, $ 0 < \theta < 1$ and a Diophantine vector $ \omega = (\omega_{n})_{n\in\mathbb{Z}}$ satisfying $
\sup_{n}|\omega_{n}| < 1$,\ then for any $\mu>2r$, sufficiently small $ \epsilon > 0 $ and some appropriate $V$, (\ref{maineq}) has a full dimensional invariant torus $ \mathcal{E}$ with amplitude in $ \mathfrak{H}_{r,\infty}$ satisfying:\\
\begin{itemize}
\item[(1)] the amplitude of $ \mathcal{E}$ is restricted as
\begin{equation*}
\frac12e^{-r|n|^{\theta}}\leq |q_n|\leq 2e^{-r|n|^{\theta}};
\end{equation*}
\item[(2)] the frequency on  $ \mathcal{E}$ was prescribed to be $(n^2+\omega_n)_{n\in\mathbb{Z}}$;\\
\item[(3)] the invariant torus  $ \mathcal{E}$ is linearly stable.
\end{itemize}
\end{thm}

\section{KAM Iteration}
\subsection{Some notations and the norm of the Hamiltonian}
 Let $q=(q_n)_{n\in\mathbb{Z}}$ and its complex conjugate $\bar {q}=(\bar q_n)_{n\in\mathbb{Z}}$. Introduce $I_n=|q_n|^2$ and $J_n=I_n-I_n(0)$ as notations but not as new variables, where $I_n(0)$ will be considered as the initial data. Then the  Hamiltonian (\ref{maineq}) has the form of
\begin{equation*}
H(q,\bar q)=N(q,\bar q)+R(q,\bar q),
\end{equation*}
where
\begin{equation*}
N(q,\bar q)=\sum_{n\in\mathbb{Z}}(n^2+V_n)|q_n|^2,
\end{equation*}
 \begin{equation*}
 R(q,\bar q)=\sum_{a,k,k'\in{\mathbb{N}^{\mathbb{Z}}}}B_{akk'}\mathcal{M}_{akk'}
 \end{equation*}
with
\begin{eqnarray}\nonumber
\mathcal{M}_{akk'}=\prod_{n\in\mathbb{Z}}I_n(0)^{a_n}q_n^{k_n}\bar q_n^{k_n'},
\end{eqnarray}
and $B_{akk'}$ are the coefficients.

Define by
\begin{eqnarray}\label{008}
\mbox{supp}\ \mathcal{M}_{akk'}=\{n:2a_n+k_n+k_n'\neq 0\},
\end{eqnarray}
%\begin{eqnarray}\label{009}
%\mbox{degree}\ \mathcal{M}_{akk'}=\sum_{n\in\mathbb{Z}}(2a_n+k_n+k_n'),
%\end{eqnarray}
and define the momentum of $\mathcal{M}_{akk'}$ by
\begin{eqnarray}\label{010}
\mbox{momentum}\ \mathcal{M}_{akk'}:=m(k,k')=\sum_{n\in\mathbb{Z}} (k_n-k_n')n.
\end{eqnarray}
Moreover,
denote by

\[
 n^{\ast}_{1}= \max\{|n|: a_{n}+k_{n}+k'_{n} \neq 0\},
\]
and
\[
m^*(k,k')=\left|m(k,k')\right|.
\]
Now we define the norm of the Hamiltonian as follows
\begin{defn}
For any given $ \rho > 0,\mu > 0$ and $0<\theta<1$, define the norm of the Hamiltonian
$R$ by
\begin{eqnarray}\label{H00}
\| R\|_{\rho,\mu}= \sup_{a,k,k'\in\mathbb{N}^{\mathbb{Z}}}\frac{\left|B_{akk'}\right|}{e^{\rho\sum_{n}(2a_{n}+k_{n}+k_{n}')|n|^{\theta}
-2\rho(n_1^\ast)^{\theta}-\mu m^\ast(k,k')^{\theta}}}.
\end{eqnarray}

\end{defn}

\subsection{Derivation of homological equations}
The proof of Theorem \ref{Thm} employs the rapidly
converging iteration scheme of Newton type to deal with small divisor problems
introduced by Kolmogorov, involving the infinite sequence of coordinate transformations.
At the $s$-th step of the scheme, a Hamiltonian
$H_{s} = N_{s} + R_{s}$
is considered, as a small perturbation of some normal form $N_{s}$. A transformation $\Phi_{s}$ is
set up so that
$$ H_{s}\circ \Phi_{s} = N_{s+1} + R_{s+1}$$
with another normal form $N_{s+1}$ and a much smaller perturbation $R_{s+1}$. We drop the index $s$ of $H_{s}, N_{s}, R_{s}, \Phi_{s}$ and shorten the index $s+1$ as $+$.

Rewrite $R$ as
\begin{equation}\label{N1}
R=R_0+R_1+R_2
\end{equation}
where
\begin{eqnarray*}
{R}_0&=&\sum_{a,k,k'\in{\mathbb{N}^{\mathbb{Z}}}\atop\mbox{supp}\ k\bigcap \mbox{supp}\ k'=\emptyset}B_{akk'}\mathcal{M}_{akk'},\\
{R}_1&=&\sum_{m\in\mathbb{Z}}J_m\left(\sum_{a,k,k'\in{\mathbb{N}^{\mathbb{Z}}}\atop\mbox{supp}\ k\bigcap \mbox{supp}\ k'=\emptyset}B_{akk'}^{(m)}\mathcal{M}_{akk'}\right),\\
{R}_2&=&\sum_{m_1,m_2\in\mathbb{Z}}J_{m_1}J_{m_2}\left(\sum_{a,k,k'\in{\mathbb{N}^{\mathbb{Z}}}\atop\mbox{no assumption}}B_{akk'}^{(m_1,m_2)}\mathcal{M}_{akk'}\right).
\end{eqnarray*}

We desire to eliminate the terms $R_0,R_1$ in (\ref{N1}) by the coordinate transformation $\Phi$, which is obtained as the time-1 map $X_F^{t}|_{t=1}$ of a Hamiltonian
vector field $X_F$ with $F=F_0+F_1$. Let ${F}_{0}$ (resp. ${F}_{1}$) has the form of ${R}_0$ (resp. ${R}_{1}$),
that is \begin{eqnarray}
\label{039}&&{F}_0=\sum_{a,k,k'\in{\mathbb{N}^{\mathbb{Z}}}\atop\mbox{supp}\ k\bigcap \mbox{supp}\ k'=\emptyset}F_{akk'}\mathcal{M}_{akk'},\\
\label{040}&&{F}_1=\sum_{m\in\mathbb{Z}}J_m\left(\sum_{a,k,k'\in{\mathbb{N}^{\mathbb{Z}}}\atop\mbox{supp}\ k\bigcap \mbox{supp}\ k'=\emptyset}F_{akk'}^{(m)}\mathcal{M}_{akk'}\right),
\end{eqnarray}
and the homological equations become
\begin{equation}\label{041}
\{N,{F}\}+R_0+R_{1}=[R_0]+[R_1],
\end{equation}
where
\begin{equation}\label{042}
[R_0]=\sum_{a\in\mathbb{N}^{\mathbb{Z}}}B_{a00}\mathcal{M}_{a00},
\end{equation}
and
\begin{equation}\label{043}
[R_1]=\sum_{m\in\mathbb{Z}}J_m\sum_{a\in\mathbb{N}^{\mathbb{Z}}}B_{a00}^{(m)}\mathcal{M}_{a00}.
\end{equation}
The solutions of the homological equations (\ref{041}) are given by
\begin{equation}\label{044}
F_{akk'}=\frac{B_{akk'}}{\sum_{n\in\mathbb{Z}}(k_n-k^{'}_n)(n^2+\widetilde{V}_n)},
\end{equation}
and
\begin{equation}\label{045}
F_{akk'}^{(m)}=\frac{B_{akk'}^{(m)}}{\sum_{n\in\mathbb{Z}}(k_n-k^{'}_n)(n^2+\widetilde{V}_n)}.
\end{equation}
The new Hamiltonian ${H}_{+}$ has the form
\begin{eqnarray}
H_{+}\nonumber&=&H\circ\Phi\\
&=&\nonumber N+\{N,F\}+R_0+R_1\\
&&\nonumber+\int_{0}^1\{(1-t)\{N,F\}+R_0+R_1,F\}\circ X_F^{t}\ \mathrm{d}{t}
+\nonumber R_2\circ X_F^1\\
&=&\label{046}N_++R_+,
\end{eqnarray}
where
\begin{equation}\label{047}
N_+=N+[R_0]+[R_1],
\end{equation}
and
\begin{equation}\label{048}
R_+=\int_{0}^1\{(1-t)\{N,F\}+R_0+R_1,F\}\circ X_F^{t}\ \mathrm{d} t+R_2\circ X_F^1.
\end{equation}
\subsection{The solvability of the homological equations (\ref{041})}In this subsection, we will estimate
the solutions of the homological equations (\ref{041}). To this end, we define the new norm for the Hamiltonian ${R}$ of the form as follows:
\begin{eqnarray}\label{049}
||{R}||_{\rho,\mu}^{+}=\max\{||R_0||_{\rho,\mu}^{+},||R_1||_{\rho,\mu}^{+}|,||R_2||_{\rho,\mu}^{+}\},
\end{eqnarray}
where
\begin{eqnarray}
\label{050}&&||R_0||_{\rho,\mu}^{+}=\sup_{a,k,k'\in\mathbb{N}^{\mathbb{Z}}}\frac{\left|B_{akk'}\right|}{e^{\rho(\sum_{n}(2a_n+k_n+k_n')|n|^{\theta}-2(n_1^*)^{\theta})-\mu m^*(k,k')^{\theta}}},\\
\label{051}&&||R_1||_{\rho,\mu}^{+}=\sup_{a,k,k'\in\mathbb{N}^{\mathbb{Z}}\atop m\in\mathbb{Z}}\frac{\left|B^{(m)}_{akk'}\right|}{e^{\rho(\sum_{n}(2a_n+k_n+k_n')|n|^{\theta}+2|m|^{\theta}-2(n_1^*)^{\theta})
-\mu m^\ast(k,k')^{\theta}}},\\
\label{052}&&||R_2||_{\rho,\mu}^{+}=\sup_{a,k,k'\in\mathbb{N}^{\mathbb{Z}}\atop
m_1,m_2\in\mathbb{Z}}\frac{\left|B^{(m_1,m_2)}_{akk'}\right|}{e^{\rho(\sum_{n}(2a_n+k_n+k_n')|n|^{\theta}
+2|m_1|^{\theta}+2|m_2|^{\theta}-2(n_1^*)^{\theta})-\mu m^\ast(k,k')^{\theta}}}.
\end{eqnarray}
Moreover, one has the following estimates:
\begin{lem}\label{N2}
Given any $ \mu>\delta>0,\rho>0$, one has
\begin{equation}\label{053}
||R||_{\rho+\delta,\mu-\delta}^{+}\leq\left(\frac{1}{\delta}\right)^{ C(\theta)\delta^{-\frac{1}{\theta}}}||R||_{\rho,\mu}
\end{equation}
and
\begin{equation}\label{054}
||R||_{\rho+\delta,\mu-\delta}\leq\frac{C(\theta)}{\delta^2}||R||_{\rho,\mu}^{+},
\end{equation}
where $C(\theta)$ is a  positive constant depending on $\theta$ only.
\end{lem}
\begin{proof}
The details of the proof will be given in the Appendix.
\end{proof}

\begin{lem}\label{N3}
Let $(\widetilde{V}_n)_{n\in\mathbb{Z}}$ be Diophantine with $\gamma>0$ (see (\ref{005})). Then for any $\rho>0,0<\delta\ll1$ (depending only on $\theta$), the solutions of the homological equations (\ref{041}), which are given by (\ref{044}) and (\ref{045}), satisfy
\begin{eqnarray}\label{061}
||{F}_i||_{\rho+\delta,\mu-2\delta}^{+}\leq \frac{1}{\gamma}\cdot e^{C(\theta)\delta^{-\frac5\theta}}||{{R_i}}||_{\rho,\mu}^{+},
\end{eqnarray}
 where $i=0,1$ and $C(\theta)$ is a positive constant depending on $\theta$ only.
\end{lem}
\begin{proof}
The details of the proof will be given in the Appendix.
\end{proof}
\subsection{The new perturbation $R_+$ and the new normal form $N_+$}Firstly, we will prove two lemmas.
\begin{lem}\label{H3}\textbf{(Poisson Bracket)}
Let $\theta\in(0,1),\rho,\mu>0$ and $0<\delta_{1},\delta_{2}\ll1$ (depending on $\theta,\rho,\mu$).\ Then one has
\begin{equation}\label{H4}
\left|\left|\{H_1,H_2\}\right|\right|_{\rho,\mu}\leq \frac{1}{\delta_{2}}\left(\frac{1}{\delta_{1}}\right)^{C({\theta}){\delta_{1}^{-\frac{1}{\theta}}}}
||H_1||_{\rho-\delta_{1},\mu+2\delta_{1}}||H_2||_{\rho-\delta_{2},\mu+2\delta_{2}},
\end{equation}
where $C(\theta)$ is a positive constant depending on $\theta$ only.
\end{lem}
\begin{proof}
The details of proof will be left in the Appendix.
\end{proof}
\begin{lem}\label{E1}
Let $\theta\in(0,1),\rho>0$ and $0<\delta_1,\delta_2\ll1$ (depending on $\theta,\rho$). Assume further \begin{equation}\label{035}
\frac{1}{\delta_2}
\left(\frac{1}{\delta_1}\right)^{C({\theta}){\delta_1^{-\frac{1}{\theta}}}}||F||_{\rho-\delta_1,\mu+2\delta_{1}} \ll 1,
\end{equation}
where $C(\theta)$ is the constant given in (\ref{H4}) in Lemma \ref{H3}.\ Then for any Hamiltonian function $H$, we get
\begin{equation}\label{036}
||H\circ\Phi_F||_{\rho,\mu}
\leq\left(1+\frac{1}{\delta_2}\left(\frac{1}{\delta_1}\right)^{C_1({\theta}){\delta_1^{-\frac{1}{\theta}}}}
||F||_{\rho-\delta_1,\mu+2\delta_{1}}\right)
||H||_{\rho-\delta_2,\mu+2\delta_{2}},
\end{equation}
where $C_1(\theta)$ is a positive constant depending only on $\theta$.
\end{lem}

\begin{proof}
 Firstly,  we expand $H\circ\Phi_F$ into the Taylor series
 \begin{equation}\label{037}
 H\circ\Phi_F=\sum_{n\geq 0}\frac{1}{n!}H^{(n)},
 \end{equation}
where $H^{(n)}=\{H^{(n-1)},F\}$ and $H^{(0)}=H$.

We will estimate $||H^{(n)}||_{\rho,\mu}$ by using Lemma \ref{H3} again and again:
\begin{eqnarray}
\nonumber||H^{(n)}||_\rho
\nonumber&=&||\{H^{(n-1)},F\}||_{\rho,\mu}\\
\nonumber&\leq&\left(\left(\frac{1}{\delta_1}\right)^{C({\theta}){\delta_1^{-\frac{1}{\theta}}}}
||F||_{\rho-\delta_1,\mu+2\delta_{1}}\right)\left(\frac{n}{\delta_2}\right)
||H^{(n-1)}||_{\rho-\frac{\delta_2}{n},\mu+\frac{2\delta_{2}}{n}}\\
\nonumber&\leq&\left(\left(\frac{1}{\delta_1}\right)^{C({\theta}){\delta_1^{-\frac{1}{\theta}}}}
||F||_{\rho-\delta_1,\mu+2\delta_{1}}\right)^2\left(\frac{n}{\delta_2}\right)^2
||H^{(n-2)}||_{\rho-\frac{2\delta_2}{n},\mu+\frac{4\delta_{2}}{n}}\\
&&\dots\nonumber\\
\label{038}&\leq&\left(\left(\frac{1}{\delta_1}\right)^{C({\theta}){\delta_1^{-\frac{1}{\theta}}}}
||F||_{\rho-\delta_1,\mu+2\delta_{1}}\right)^n\left(\frac{n}{\delta_2}\right)^n
||H||_{\rho-\delta_2,\mu+2\delta_{2}}.
\end{eqnarray}
Hence in view of (\ref{038}), one has
\begin{eqnarray*}
||H\circ\Phi_F||_{\rho,\mu}
&\leq&\sum_{n\geq 0}\frac{1}{n!}\left(\left(\frac{1}{\delta_1}\right)^{C({\theta}){\delta_1^{-\frac{1}{\theta}}}}
||F||_{\rho-\delta_1,\mu+2\delta_{1}}\right)^{n}\left(\frac{n}{\delta_2}\right)^n||H||_{\rho-\delta_2,\mu+2\delta_{2}}\\
&=&\sum_{n\geq 0}\frac{n^n}{n!}\left(\frac1{\delta_2}\left(\frac{1}{\delta_1}\right)^{C({\theta}){\delta_1^{-\frac{1}{\theta}}}}
||F||_{\rho-\delta_1,\mu+2\delta_{1}}\right)^{n}||H||_{\rho-\delta_2,\mu+2\delta_{2}}\\
&\leq&\sum_{n\geq 0}\left(\frac e{\delta_2}\left(\frac{1}{\delta_1}\right)^{C({\theta}){\delta_1^{-\frac{1}{\theta}}}}
||F||_{\rho-\delta_1,\mu+2\delta_{1}}\right)^{n}||H||_{\rho-\delta_2,\mu+2\delta_{2}}\\
&&(\mbox{in view of $n^n<n!e^n$)}\\
&\leq & \left(1+\frac{1}{\delta_2}\left(\frac{1}{\delta_1}\right)^{C({\theta}){\delta_1^{-\frac{1}{\theta}}}}
||F||_{\rho-\delta_1,\mu+2\delta_{1}}\right)
||H||_{\rho-\delta_2,\mu+2\delta_{2}}\\
 &&\mbox{(in view of (\ref{036}) and $0<\delta_1,\delta_2\ll1$)},
\end{eqnarray*}
where $C_1(\theta)$ is a positive constant depending on $\theta$ only.
\end{proof}
Recall the new term $R_+$ is given by (\ref{048}) and
%, i.e. $R_+=R_2+\mathcal{R}$, where
%\begin{equation*}
%\mathcal{R}=\int_{0}^1\{(1-t)\{N,F\}+R_0+R_1,F\}\circ X_F^{t}dt+\int_0^1\{R_2,F\}\circ X_F^tdt.
%\end{equation*}
%Firstly, for $i=0,1$, one has
%\begin{eqnarray}
%\nonumber||{F}_i||_{\rho+\delta,\mu-\frac{3}{2}\delta}
%&\leq&\nonumber\frac{C_1(\theta)}{{\delta^2}}||F_i||_{\rho+\frac{\delta}2,\mu-\delta}^+\qquad\qquad\qquad\qquad\mbox{(in view of  (\ref{054}))}\\
%&\leq&\nonumber \frac{C_1(\theta)}{{\delta^2}}\cdot \frac{e^3}{\gamma}\cdot e^{{C_2({\theta)}}{\delta^{-\frac5{\theta}}}}||R_i||_{\rho,\mu}^+\qquad\mbox{(in view of  (\ref{061}))}  \\
%&\leq&\label{068}\frac{1}{\gamma\delta^3}e^{{\delta^{-\frac6{\theta}}}}||R_i||_{\rho,\mu}^+\qquad \ \qquad\qquad\quad \mbox{(by assuming $\delta\ll 1$)}.
%\end{eqnarray}
write
\begin{equation}\label{069}
R_+=R_{0+}+R_{1+}+R_{2+}.
\end{equation}
Following the proof of CLSY, one has
\begin{eqnarray}
||R_{0+}||_{\rho+3\delta,\mu-\frac{11}{2}\delta}^{+}
\label{0861} &\leq& \frac1{\gamma}\cdot e^{{\delta^{-\frac{10}{\theta}}}}(||R_0||_{\rho,\mu}^++||R_1||_{\rho,\mu}^+)(||R_0||_{\rho,\mu}^+
+{||R_1||_{\rho,\mu}^+}^2),\\
||R_{1+}||_{\rho+3\delta,\mu-\frac{11}{2}\delta}^{+}
\label{0862}&\leq& \frac1{\gamma}\cdot e^{{\delta^{-\frac{10}{\theta}}}}(||R_0||_{\rho,\mu}^++{||R_1||_{\rho,\mu}^+}^2),\\
||R_{2+}||_{\rho+3\delta,\mu-\frac{11}{2}\delta}^{+}
\label{0863}&\leq& ||R_2||_{\rho,\mu}^++\frac1{\gamma}\cdot e^{{\delta^{-\frac{10}{\theta}}}}(||R_0||_{\rho,\mu}^++||R_1||_{\rho,\mu}^+).
\end{eqnarray}

The new normal form $N_+$ is given in (\ref{047}). Note that $[R_0]$ (in view of (\ref{042})) is a constant which does not affect the Hamiltonian vector field. Moreover, in view of (\ref{043}), we denote by
\begin{equation}\label{087}
\omega_{n+}=n^2+ \widetilde{V}_n+\sum_{a\in\mathbb{N}^{\mathbb{Z}}}B_{a00}^{(n)}\mathcal{M}_{a00},
\end{equation}
where the terms $\sum_{a\in\mathbb{N}^{\mathbb{Z}}}B_{a00}^{(n)}\mathcal{M}_{a00}$ is the so-called frequency shift. The estimate of $\left|\sum_{a\in\mathbb{N}^{\mathbb{Z}}}B_{a00}^{(n)}\mathcal{M}_{a00}\right|$ will be given in the next section (see (\ref{114}) for the details).

Finally, we give the estimate of the Hamiltonian vector field.
\begin{lem}\label{H6}
Given a Hamiltonian
\begin{equation}\label{028}
H=\sum_{a,k,k'\in\mathbb{N}^{\mathbb{Z}}}B_{akk'}\mathcal{M}_{akk'},
\end{equation}
then for any $\mu>r>(\frac{1}{2-2^{\theta}}+3)\rho, ||q||_{r,\infty} < 1$ and $ ||I(0)||_{r,\infty} < 1 $,\ one has
\begin{equation}\label{029}
\sup_{j \in \mathbb{Z}}\left|e^{r|j|^{\theta}}\frac{\partial{H}}{\partial q_{j}}\right|\leq C(r,\rho,\mu,\theta)||H||_{\rho,\mu},
\end{equation}
where $C(r,\rho,\mu,\theta)$ is a positive constant depending on $r,\rho,\mu$ and $\theta$ only,\ and
\begin{equation}\label{030}
||I(0)||_{r,\infty} := \sup_{n\in \mathbb{Z}}|I_{n}(0)|e^{2r|n|^{\theta}}.
\end{equation}
\end{lem}
\begin{proof}
The details of the proof will be given in the Appendix.
\end{proof}

\section{Iteration and Convergence}

Now we give the precise
set-up of iteration parameters. Let $s\geq1$ be the $s$-th KAM
step.
 \begin{itemize}
 \item[] $\rho_0=\rho,$\ $r\geq \frac{100\rho}{2-2^{\theta}}$, $\mu_0\geq2r$,
 \item[]$\delta_{s}=\frac{\rho}{s^2}$,

 \item[]$\rho_{s+1}=\rho_{s}+3\delta_s$,\item[] $\mu_{s+1}=\mu_{s}-6\delta_s$

 \item[]$\epsilon_s=\epsilon_{0}^{(\frac{3}{2})^s}$, which dominates the size of
 the perturbation,

 \item[]$\lambda_s=e^{-C(\theta)(\ln{\frac{1}{\epsilon_{s+1}}})^{\frac{4}{\theta+4}}}$,

 \item[]$\eta_{s+1}=\frac{1}{20}\lambda_s\eta_s$,

 \item[]$d_0=0,\,d_{s+1}=d_s+\frac{1}{\pi^2(s+1)^2}$,

 \item[]$D_s=\{(q_n)_{n\in\mathbb{Z}}:\frac{1}{2}+d_s\leq|q_n|e^{r|n|^{\theta}}\leq1-d_s\}$.
 \end{itemize}
Denote the complex cube of size $\lambda>0$:
\begin{equation}\label{088}
\mathcal{C}_{\lambda}({V^*})=\left\{(V_n)_{n\in\mathbb{Z}}\in\mathbb{C}^{\mathbb{Z}}:|V_n-V^*_n|\leq \lambda\right\}.
\end{equation}

\begin{lem}{\label{E2}}
Suppose $H_{s}=N_{s}+R_{s}$ is real analytic on $D_{s}\times\mathcal{C}_{\eta_{s}}(V^*_{s})$,
where $$N_{s}=\sum_{n\in\mathbb{Z}}(n^2+\widetilde V_{n,s})|q_n|^2$$ is a normal form with coefficients satisfying
\begin{eqnarray}
\label{089}&&\widetilde{V}_{s}(V_{s}^*)=\omega,\\
\label{090}&&\left|\left|\frac{\partial \widetilde{V}_s}{{\partial V}}-I\right|\right|_{l^{\infty}\rightarrow l^{\infty}}<d_s\epsilon_{0}^{\frac{1}{10}},
\end{eqnarray}
and $R_{s}=R_{0,s}+R_{1,s}+R_{2,s}$ satisfying
\begin{eqnarray}
\label{091}&&||R_{0,s}||_{\rho_{s},\mu_{s}}^{+}\leq \epsilon_{s},\\
\label{092}&&||R_{1,s}||_{\rho_{s},\mu_{s}}^{+}\leq \epsilon_{s}^{0.6},\\
\label{093}&&||R_{2,s}||_{\rho_{s},\mu_{s}}^{+}\leq (1+d_s)\epsilon_0.
\end{eqnarray}
Then for all $V\in\mathcal{C}_{\eta_{s}}(V_{s}^*)$ satisfying $\widetilde V_{s}(V)\in\mathcal{C}_{\lambda_s}(\omega)$, there exist real analytic symplectic coordinate transformations
$\Phi_{s+1}:D_{s+1}\rightarrow D_{s}$ satisfying
\begin{eqnarray}
\label{094}&&||\Phi_{s+1}-id||_{r,\infty}\leq \epsilon_{s}^{0.5},\\
\label{095}&&||D\Phi_{s+1}-I||_{(r,\infty)\rightarrow(r,\infty)}\leq \epsilon_{s}^{0.5},
\end{eqnarray}
such that for
$H_{s+1}=H_{s}\circ\Phi_{s+1}=N_{s+1}+R_{s+1}$, the same assumptions as above are satisfied with `$s+1$' in place of `$s$', where $\mathcal{C}_{\eta_{s+1}}(V_{s+1}^*)\subset\widetilde V_{s}^{-1}(\mathcal{C}_{\lambda_s}(\omega))$ and
\begin{equation}\label{096}
||\widetilde{V}_{s+1}-\widetilde{V}_{s}||_{\infty}\leq\epsilon_{s}^{0.5},
\end{equation}
\begin{equation}\label{097}
||V_{s+1}^*-V_{s}^*||_{\infty}\leq2\epsilon_{s}^{0.5}.
\end{equation}
\end{lem}

\begin{proof}
In the step $s\rightarrow s+1$, there is saving of a factor
\begin{equation}\label{19010501}
e^{-\delta_{s}\left(\sum_{n}(2a_n+k_n+k'_n)|n|^{\theta}-2|n_1^*|^{\theta}+2m^{\ast}(k,k')^{\theta}\right)}.
\end{equation}
By (\ref{H2}), one has
\begin{equation}
(\ref{19010501})\leq e^{-(2-2^\theta)\delta_{s}\left(\sum_{i\geq3}|n_i|^{\theta}\right)-\delta_sm^*(k,k')^{\theta}}\leq e^{-(2-2^\theta)\delta_{s}\left(\sum_{i\geq3}|n_i|^{\theta}+m^*(k,k')^{\theta}\right)}.
\end{equation}
Recalling after this step, we need
\begin{eqnarray*}
&&||R_{0,s+1}||_{\rho_{s+1},\mu_{s+1}}^{+}\leq \epsilon_{s+1},\\
&&||R_{1,s+1}||_{\rho_{s+1},\mu_{s+1}}^{+}\leq \epsilon_{s+1}^{0.6}.
\end{eqnarray*}
Consequently, in $R_{i,s}\ (i=0,1)$, it suffices to eliminate the nonresonant monomials $\mathcal{M}_{akk'}$ for which
\begin{equation*}
e^{-(2-2^\theta)\delta_{s}(\sum_{i\geq3}|n_i|^{\theta}+m^{\ast}(k,k')^{\theta})}\geq\epsilon_{s+1},
\end{equation*}
that is
\begin{equation}\label{098}
\sum_{i\geq3}|n_i|^{\theta}+m^{\ast}(k,k')^{\theta}\leq\frac{s^2}{(2-2^\theta)\rho}\ln\frac{1}{\epsilon_{s+1}}.
\end{equation}
On the other hand, in the small divisors analysis (see Lemma \ref{a1}), one has
\begin{eqnarray}
\nonumber\sum_{n\in\mathbb{Z}}|k_n-k_n'||n|^{\theta/2}
\nonumber&\leq& 3\cdot 6^{\theta/2}\left(\sum_{i\geq3}|n_i|^{\theta}+m^{\ast}(k,k')^{\theta}\right)\ \  \\
&\leq&\nonumber\frac{3\cdot 6^{\theta/2}\cdot s^2}{(2-2^\theta)\rho} \ln\frac{1}{\epsilon_{s+1}}\ \  \mbox{(in view of  (\ref{098}))}\\
\nonumber&:=& B_s.
\end{eqnarray}
Hence we need only impose condition on $(\widetilde{V}_n)_{|n|\leq \mathcal{N}_{s}}$, where
\begin{equation}\label{100}
\mathcal{N}_{s}\sim B_s^{2/\theta}.
\end{equation}
Correspondingly, the Diophantine condition becomes
\begin{equation}\label{101}
\left|\left|\sum_{|n|\leq \mathcal{N}_{s}}(k_n-k'_n)\widetilde{V}_{n,s}\right|\right|\geq \gamma\prod_{|n|\leq \mathcal{N}_{s}}\frac{1}{1+(k_n-k'_n)^2|n|^4}.
\end{equation}
We finished the truncation step.

Next we will show (\ref{101})  preserves under small perturbation of $(\widetilde{V}_n)_{|n|\leq \mathcal{N}_{s}}$ and this is equivalent to get lower bound on the right hand side of (\ref{101}). Let \begin{equation}M_s\sim \left(\frac{B_s}{\ln B_s}\right)^{\frac{2.5}{\theta+2.5}},\end{equation} then we have
\begin{eqnarray}
\nonumber\prod_{|n|\leq \mathcal{N}_{s}}\frac{1}{1+(k_n-k'_n)^2|n|^4}
\nonumber&=&e^{\sum_{|n|\leq M_s}\ln\left(\frac{1}{1+(k_n-k'_n)^2|n|^4}\right)+\sum_{|n|> M_s}\ln\left(\frac{1}{1+(k_n-k'_n)^2|n|^4}\right)}\\
\nonumber&\geq& e^{-C(\theta)\sum_{|n|\leq M_s,k_n\neq k'_n}\ln\left(|k_n-k'_n||n|^{\frac\theta2}\right)-\sum\limits_{|n|> M_s,k_n\neq k'_n}4\left(|k_n-k'_n|\cdot\ln|n|\right)}\\
\nonumber&\geq& e^{-C(\theta)M_s\ln B_s-4\sum_{|n|> M_s,k_n\neq k'_n}\left(|k_n-k'_n||n|^{\frac{\theta}{2}}(|n|^{-\frac{\theta}{2}}\ln|n|)\right)}\\
\nonumber&\geq& e^{-C(\theta)M_s\ln B_s-C(\theta)(M_s^{-\frac{\theta}{2}}\ln M_s)B_s}\\
\nonumber&\geq& e^{-C(\theta)M_s\ln B_s-C(\theta)M_s^{-\frac{\theta}{2.5}}B_s}\\
\nonumber&\geq& e^{-C(\theta)B^{\frac{3}{\theta+3}}}\\
\nonumber&\geq& e^{-C(\theta)s^{\frac{6}{3+\theta}} (\ln{\frac{1}{\epsilon_{s+1}}})^{\frac{3}{\theta+3}}}\\
\label{103}&>&e^{-C(\theta)(\ln{\frac{1}{\epsilon_{s+1}}})^{\frac{4}{\theta+4}}}=\lambda_s,
\end{eqnarray}
where the last inequality is based on $\epsilon_0$ is small enough.

Assuming $V\in \mathcal{C}_{\lambda_s}(\omega)$, from the lower bound (\ref{103}), the relation (\ref{101}) remains true if we substitute $V$ for $\omega$. Moreover, there is analyticity on $\mathcal{C}_{\lambda_s}(\omega)$. The transformations $\Phi_{s+1}$ is obtained as the time-1 map $X_{F_s}^{t}|_{t=1}$ of the Hamiltonian
vector field $X_{F_s}$ with $F_s=F_{0,s}+F_{1,s}$. Taking $\rho=\rho_s$, $\delta=\delta_s$ in Lemma \ref{N3}, we get
\begin{eqnarray}\label{104}
||F_{i,s}||_{\rho_s+\delta_s,\mu_s-{2}\delta_s}^{+}\leq \frac{1}{\gamma}\cdot e^{C(\theta)\delta_s^{-\frac5\theta}}||R_{i,s}||_{\rho_s,\mu_s}^{+},
\end{eqnarray}
where $i=0,1$. By Lemma \ref{N2}, we get
\begin{equation}\label{105}
||F_{i,s}||_{\rho_s+2\delta_s,\mu_s-3\delta_s}\leq\frac{C(\theta)}{\delta_s^2}||F_{i,s}||
_{\rho_s+\delta_s,\mu_s-2\delta_s}^{+}.
\end{equation}
Combining (\ref{091}), (\ref{092}), (\ref{104}) and (\ref{105}), we get
\begin{equation}\label{106}
||F_{s}||_{\rho_s+2\delta_s,\mu_s-3\delta_s}\leq\frac{C(\theta)}{\gamma\delta_s^2}e^{C(\theta)\delta_s^{-\frac5\theta}}(\epsilon_{s}+\epsilon_{s}^{0.6}).
\end{equation}
By Lemma \ref{H6}, we get
\begin{eqnarray}
\sup_{||q||_{r,\infty}<1}||X_{F_s}||_{r,\infty}\nonumber
&\leq&C(\rho,\theta)||F_{s}||_{\rho_s+2\delta_s,\mu_{s}-3\delta_{s}}\nonumber\\
&\leq&\frac{C(\rho,\theta)}{\gamma\delta_s^2}e^{C(\theta)\delta_s^{-\frac5\theta}}(\epsilon_{s}+\epsilon_{s}^{0.6})
\nonumber\\
\label{107}&\leq&\epsilon_{s}^{0.55},
\end{eqnarray}
where noting that $0<\epsilon_0\ll1$ small enough and depending on $\rho,\theta$ only.

Since $\epsilon_{s}^{0.55}\ll\frac{1}{\pi^2(s+1)^2}=d_{s+1}-d_s$, we have $\Phi_{s+1}:D_{s+1}\rightarrow D_{s}$ with
\begin{equation}\label{108}
\|\Phi_{s+1}-id\|_{r,\infty}\leq\sup_{q\in D_s}\|X_{F_s}\|_{r,\infty}\leq\epsilon_{s}^{0.55}<\epsilon_{s}^{0.5},
\end{equation}
which is the estimate (\ref{094}). Moreover, from (\ref{108}) we get
\begin{equation}\label{109}
\sup_{q\in D_s}||DX_{F_s}-I||_{r,\infty}\leq\frac{1}{d_s}\epsilon_{s}^{0.55}\ll\epsilon_{s}^{0.5},
\end{equation}
and thus the estimate (\ref{095}) follows.

Moreover, under the assumptions (\ref{091})-(\ref{093}) at stage $s$, we get from (\ref{0861}), (\ref{0862}) and (\ref{0863}) that
\begin{eqnarray*}
||R_{0,s+1}||_{\rho_{s+1},\mu_{s+1}}^{+}
&\leq& e^{\frac{s^{\frac{20}{\theta}}}{\rho^{\frac{10}{\theta}}}}
\left(\epsilon_{0}^{(\frac{3}{2})^s}+\epsilon_{0}^{0.9(\frac{3}{2})^{s-1}}\right)\left(\epsilon_{0}^{(\frac{3}{2})^s}+\epsilon_{0}^{1.8(\frac{3}{2})^{s-1}}\right)\\
&=&e^{\frac{s^{\frac{20}{\theta}}}{\rho^{\frac{10}{\theta}}}}\left(\epsilon_{0}^{2.2(\frac{3}{2})^s}+\epsilon_{0}^{1.8(\frac{3}{2})^s}
+\epsilon_{0}^{1.6(\frac{3}{2})^s}+\epsilon_{0}^{2(\frac{3}{2})^s}\right)\\
&\leq& 4e^{\frac{s^{\frac{20}{\theta}}}{\rho^{\frac{20}{\theta}}}}\epsilon_{0}^{1.6(\frac{3}{2})^s}\\
&<&\epsilon_{0}^{1.5(\frac{3}{2})^s}\ \mbox{for $0<\epsilon_0\ll1$ (depending on $\rho,\theta$ only)}\\
&=&\epsilon_{s+1},\\
||R_{1,s+1}||_{\rho_{s+1},\mu_{s+1}}^{+}
&\leq& e^{\frac{s^{\frac{20}{\theta}}}{\rho^{\frac{20}{\theta}}}}\left( \epsilon_{0}^{(\frac{3}{2})^s}+\epsilon_{0}^{1.8(\frac{3}{2})^{s-1}}\right)\\
&=&e^{\frac{s^{\frac{20}{\theta}}}{\rho^{\frac{20}{\theta}}}}\left( \epsilon_{0}^{(\frac{3}{2})^s}+\epsilon_{0}^{1.2(\frac{3}{2})^{s}}\right)\\
&\leq&2 e^{\frac{s^{\frac{20}{\theta}}}{\rho^{\frac{20}{\theta}}}}\epsilon_{0}^{(\frac{3}{2})^s}\\
&<&\epsilon_{s+1}^{0.6}\ \ \mbox{for $0<\epsilon_0\ll1$ (depending on $\rho,\theta$ only)},\\
\end{eqnarray*}
and
\begin{eqnarray*}
||R_{2,s+1}||_{\rho_{s+1},\mu_{s+1}}^{+}
&\leq& ||R_{2,s}||_{\rho_{s},\mu_{s}}^{+}+e^{\frac{s^{\frac{20}{\theta}}}{\rho^{\frac{20}{\theta}}}}
\left(\epsilon_{0}^{(\frac{3}{2})^s}+\epsilon_{0}^{0.6(\frac{3}{2})^{s}}\right)\\
&\leq&(1+d_s)\epsilon_0+2e^{\frac{s^{\frac{20}{\theta}}}{\rho^{\frac{20}{\theta}}}}
\epsilon_{0}^{0.6(\frac{3}{2})^s}\\
&\leq&(1+d_{s+1})\epsilon_0\ \ \mbox{for $0<\epsilon_0\ll1$ (depending on $\rho,\theta$ only)},
\end{eqnarray*}
which are just the assumptions (\ref{091})-(\ref{093}) at stage $s+1$.

If $V\in \mathcal{C}_{\frac{\eta_s}{2}}(V_s^*)\subset\mathcal{C}_{{\eta_s}}(V_s^*)$ and using Cauchy's estimate, for any $m$ one has
\begin{eqnarray}
\nonumber\sum_{n\in\mathbb{Z}}\left|\frac{\partial \widetilde{V}_{m,s}}{\partial V_n}(V)\right|
\nonumber&\leq& \frac{2}{\eta_s}||\widetilde{V}_s||_\infty\\
\label{110}&<&10 \eta_s^{-1}\ \ \mbox{(since $||\widetilde{V}_s||_\infty\leq 1 $)}.
\end{eqnarray}
Let $V\in \mathcal{C}_{\frac{1}{10}\lambda_s\eta_s}(V_s^*)$, then
\begin{eqnarray*}
||\widetilde{V}_s(V)-\omega||_{\infty}
&=&||\widetilde{V}_s(V)-\widetilde{V}_s(V_s^*)||_{\infty}\\
& \leq&\sup_{\mathcal{C}_{\frac{1}{10}\lambda_s\eta_s}(V_s)}\left|\left|\frac{\partial \widetilde{V}_s}{\partial V}\right|\right|_{l^{\infty}\rightarrow l^{\infty}}||V-V_s^*||_{\infty}\\
&<&10 \eta_s^{-1}\cdot\frac{1}{10}\lambda_s\eta_s\ \ \mbox{(in view of (\ref{110}))}\\
&=&\lambda_s,
\end{eqnarray*}
that is
\begin{equation*}
\widetilde{V}_s\left(\mathcal{C}_{\frac{1}{10}\lambda_s\eta_s}(V_s)\right)\subseteq \mathcal{C}_{\lambda_s}(\omega).
\end{equation*}
Note that
\begin{eqnarray}
\nonumber\left|B^{(m)}_{a00}\right|
\nonumber&\leq& ||R_{1,s+1}||_{\rho_{s+1},\mu_{s+1}}^+e^{2\rho_{s+1}\left(\sum_{n}a_n|n|^{\theta}+|m|^{\theta}-(n_1^{*})^{\theta}\right)}\\
\label{111}&<&\epsilon_{0}^{0.6(\frac{3}{2})^{s}}e^{2\rho_{s+1}\left(\sum_{n}a_n|n|^{\theta}+|m|^{\theta}-(n_1^{*})^{\theta}\right)}.
\end{eqnarray}
Assuming further
\begin{equation}\label{112}
I_{n}(0)\leq e^{-2r|n|^{\theta}}
\end{equation}
and for any $s$,
\begin{equation}\label{113}
\rho_s<\frac{1}{2}r,
\end{equation}
we obtain
\begin{eqnarray}
\nonumber\left|\sum_{a\in\mathbb{N}^{\mathbb{Z}}}B^{(m)}_{a00}\mathcal{M}_{a00}\right|
\nonumber&\leq & \epsilon_{0}^{0.6(\frac{3}{2})^{s}}\sum_{a\in\mathbb{N}^{\mathbb{Z}}}e^{2\rho_{s+1}\left(\sum_{n}a_n|n|^{\theta}+|m|^{\theta}-(n_1^{*})^{\theta}\right)}\prod_{n\in\mathbb{Z}}I_{n}(0)^{a_n}\\
\nonumber&\leq& \epsilon_{0}^{0.6(\frac{3}{2})^{s}}\sum_{a\in\mathbb{N}^{\mathbb{Z}}}e^{2\rho_{s+1}\left(\sum_{n}a_n|n|^{\theta}\right)}\prod_{n\in\mathbb{Z}}I_{n}(0)^{a_n}\\
\nonumber&\leq& \epsilon_{0}^{0.6(\frac{3}{2})^{s}}\sum_{a\in\mathbb{N}^{\mathbb{Z}}}e^{\sum_{n}2\rho_{s+1}a_n|n|^{\theta}-\sum_{n}2r a_n|n|^{\theta}}\ \mbox{(in view of (\ref{112}))}\\
\nonumber&\leq& \epsilon_{0}^{0.6(\frac{3}{2})^{s}}\sum_{a\in\mathbb{N}^{\mathbb{Z}}}e^{-r\left(\sum_{n}a_n|n|^{\theta}\right)}\ \mbox{(in view of (\ref{113}))}\\
\nonumber&\leq& \epsilon_{0}^{0.6(\frac{3}{2})^{s}}\prod_{n\in\mathbb{Z}}\left(1-e^{-r |n|^{\theta}}\right)^{-1}\\% \ \mbox{(by Lemma \ref{a2})}\\
\label{114}&\leq&\left(\frac{1}{r}\right)^{C(\theta){r^{-\frac{1}{\theta}}}}\epsilon_{0}^{0.6(\frac{3}{2})^{s}}.%\ \ \mbox{(by Lemma \ref{a4})}.
\end{eqnarray}
By (\ref{114}), we have
\begin{eqnarray}
\nonumber\left|\widetilde{V}_{m,s+1}-\widetilde{V}_{m,s}\right|
\nonumber&<&\left(\frac{1}{r}\right)^{C(\theta){r^{-\frac{1}{\theta}}}}\epsilon_{0}^{0.6(\frac{3}{2})^{s}}\\
\label{115}&<&\epsilon_{s}^{0.5}\ \ \mbox{(for $\epsilon_0$ small enough)},
\end{eqnarray}
which verifies (\ref{096}). Further applying Cauchy's estimate on $\mathcal{C}_{\lambda_s\eta_s}(V_s^*)$, one gets
\begin{eqnarray}
\nonumber\sum_{n\in\mathbb{Z}}\left|\frac{\partial \widetilde{V}_{m,s+1}}{\partial V_n}-\frac{\partial \widetilde{V}_{m,s}}{\partial V_n}\right|
\nonumber&\leq& C(\theta)\frac{||\widetilde{V}_{s+1}-\widetilde{V}_{s}||_\infty}{\lambda_s\eta_s}\\
\nonumber&\leq& C(\theta)\frac{\epsilon_{s}^{0.5}}{\lambda_s\eta_s}\\
\nonumber&\leq& e^{C(\theta)(\ln\frac{1}{\epsilon_{s+1}})^{\frac{4}{4+\theta}}-\frac13\ln\frac{1}{\epsilon_{s+1}}}\left(\frac{1}{\eta_s}\right)\\
\nonumber&\leq& e^{-\frac14\ln\frac{1}{\epsilon_{s+1}}}\left(\frac{1}{\eta_s}\right)\ \ \mbox{(for $\epsilon_0$ small enough)}\\
\label{116}&=& \frac{1}{\eta_s}\epsilon_{0}^{\frac{1}{4}(\frac{3}{2})^{s+1}}.
\end{eqnarray}
Since
\begin{equation*}
\eta_{s+1}=\frac{1}{20}\lambda_s\eta_s,
\end{equation*}
it follows that
\begin{eqnarray}
\nonumber\eta_{s+1}&\geq& \eta_s e^{-C(\theta)(\ln\frac{1}{\epsilon_0})^{\frac{4}{4+\theta}}(\frac32)^{\frac{4}{4+\theta}(s+1)}}\\
\nonumber&\geq& \eta_se^{-C(\theta)\ln\frac{1}{\epsilon_0}\cdot(\frac32)^{\frac{5}{5+\theta}s}}\ \ \ \mbox{(for $\epsilon_0$ small enough)}\\
\label{117}&=&\eta_s\epsilon_0^{C(\theta)(\frac32)^{\frac{5s}{5+\theta}}},
\end{eqnarray}
and hence by iterating (\ref{117}) implies
\begin{eqnarray}
\nonumber\eta_{s}&\geq&\eta_0\epsilon_{0}^{C(\theta)\sum_{i=0}^{s-1}(\frac{3}{2})^{\frac{5i}{\theta+5}}}\\
\nonumber&=&\eta_0\epsilon_{0}^{C(\theta)\frac{(\frac32)^{\frac{5s}{\theta+5}}-1}{(\frac32)^{\frac{5}{\theta+5}}-1}}\\
\nonumber&>&\epsilon_{0}^{C(\theta)(\frac{3}{2})^{\frac{5s}{\theta+5}}}\\
\label{118}&\geq&\epsilon_{0}^{\frac{1}{100}(\frac{3}{2})^{s}}\ \ \ \mbox{(for $\epsilon_0$ small enough)}.
\end{eqnarray}
On $ \mathcal{C}_{\frac{1}{10}\lambda_s\eta_s}(V_s^*)$ and for any $m$, we deduce from (\ref{116}), (\ref{118}) and the assumption (\ref{090}) that
\begin{eqnarray*}
\sum_{n\in\mathbb{Z}}\left|\frac{\partial \widetilde{V}_{m,s+1}}{\partial V_n}-\delta_{mn}\right|
&\leq&\sum_{n\in\mathbb{Z}}\left|\frac{\partial \widetilde{V}_{m,s+1}}{\partial V_n}-\frac{\partial \widetilde{V}_{m,s}}{\partial V_n}\right|+\sum_{n\in\mathbb{Z}}\left|\frac{\partial \widetilde{V}_{m,s}}{\partial V_n}-\delta_{mn}\right|\\
&\leq&\epsilon_{0}^{(\frac{3}{8}-\frac{1}{100})(\frac{3}{2})^{s}}+d_s\epsilon_{0}^{\frac{1}{10}}\\
&<&d_{s+1}\epsilon_{0}^{\frac{1}{10}},
\end{eqnarray*}
and consequently
\begin{equation}\label{119}
\left|\left|\frac{\partial \widetilde{V}_{s+1}}{{\partial V}}-I\right|\right|_{l^{\infty}\rightarrow l^{\infty}}<d_{s+1}\epsilon_{0}^{\frac{1}{10}},
\end{equation}
which verifies (\ref{090}) for $s+1$.

Finally, we will freeze $\omega$ by invoking an inverse function theorem. Consider the following functional equation
\begin{equation}\label{120}
\widetilde{V}_{s+1}(V_{s+1}^*)=\omega,  V_{s+1}^*\in \mathcal{C}_{\frac{1}{10}\lambda_s\eta_s}(V_s^*),
\end{equation}
from (\ref{119}) and the standard inverse function theorem implies (\ref{120}) having a solution $V_{s+1}^*$, which verifies (\ref{089}) for $s+1$. Rewriting (\ref{120}) as
\begin{equation}\label{121}
V_{s+1}^*-V_s^*=(I-\widetilde{V}_{s+1})(V^*_{s+1})-(I-\widetilde{V}_{s+1})({V^*_s})+(\widetilde{V}_s-\widetilde{V}_{s+1})(V_s^*),
\end{equation}
and  using (\ref{115}) (\ref{119}) implies
\begin{equation}\label{122}
||V_{s+1}^*-V_s^*||_{\infty}\leq (1+d_{s+1})\epsilon_{0}^{\frac{1}{10}}||V_{s+1}^*-V_s^*||_{\infty}+\epsilon_s^{0.5}<2\epsilon_s^{0.5}\ll \lambda_s\eta_s,
\end{equation}
which verifies (\ref{097}) and completes the proof of the iterative lemma.
\end{proof}

We are now in a position to prove the convergence. To apply iterative lemma with $s=0$, set
\begin{equation*}
V_0=\omega,\hspace{12pt}\widetilde{V}_0=id,\hspace{12pt}\eta_0=1-\sup_{n\in\mathbb{Z}}|\omega_n|,
\hspace{12pt}r=20\rho_0=\mu_0,\hspace{12pt}\epsilon_0=C\epsilon,
\end{equation*}
and consequently (\ref{089})--(\ref{093}) with $s=0$ are satisfied. Hence, the iterative lemma applies, and we obtain a decreasing
sequence of domains $D_{s}\times\mathcal{C}_{\eta_{s}}(V_{s}^*)$ and a sequence of
transformations
\begin{equation*}
\Phi^s=\Phi_1\circ\cdots\circ\Phi_s:\hspace{6pt}D_{s}\times\mathcal{C}_{\eta_{s}}(V_{s}^*)\rightarrow D_{0}\times\mathcal{C}_{\eta_{0}}(V_{0}^*),
\end{equation*}
such that $H\circ\Phi^s=N_s+P_s$ for $s\geq1$. Moreover, the
estimates (\ref{094})--(\ref{097}) hold. Thus we can show $V_s^*$ converge to a limit $V_*$ with the estimate
\begin{equation*}
||V_*-\omega||_{\infty}\leq\sum_{s=0}^{\infty}2\epsilon_{s}^{0.5}<\epsilon_{0}^{0.4},
\end{equation*}
and $\Phi^s$ converge uniformly on $D_*\times\{V_*\}$, where $D_*=\{(q_n)_{n\in\mathbb{Z}}:\frac{2}{3}\leq|q_n|e^{r|n|^{\theta}}\leq\frac{5}{6}\}$, to $\Phi:D_*\times\{V_*\}\rightarrow D_0$ with the estimates
\begin{eqnarray}
\nonumber&&||\Phi-id||_{r,\infty}\leq \epsilon_{s}^{0.4},\\
\nonumber&&||D\Phi-I||_{(r,\infty)\rightarrow(r,\infty)}\leq \epsilon_{s}^{0.4}.
\end{eqnarray}
Hence
\begin{equation}\label{123}
H_*=H\circ\Phi=N_*+R_{2,*},
\end{equation}
where
\begin{equation}\label{124}
N_*=\sum_{n\in\mathbb{Z}}(n^2+\omega_n)|q_n|^2
\end{equation}
and
\begin{equation}\label{125}
||R_{2,*}||_{\frac{r}{2},\frac{r}{2}}^{+}\leq\frac{7}{6}\epsilon_0.
\end{equation}
By (\ref{029}), the Hamiltonian vector field $X_{R_{2,*}}$ is a bounded map from $\mathfrak{H}_{r,\infty}$ into $\mathfrak{H}_{r,\infty}$. Taking
\begin{equation}\label{126}
I_n(0)=\frac{3}{4}e^{-2r|n|^{\theta}},
\end{equation}
we get an invariant torus $\mathcal{T}$ with frequency $(n^2+\omega_n)_{n\in\mathbb{Z}}$ for ${X}_{H_*}$. Finally, by $X_H\circ\Phi=D\Phi\cdot{X}_{H_*}$, $\Phi(\mathcal{T})$ is the desired invariant torus for the NLS (\ref{maineq}). Moreover, we deduce the torus $\Phi(\mathcal{T})$ is linearly stable from the fact that (\ref{123}) is a normal form of order 2 around the invariant torus.

\section{Appendix}
\subsection{Technical Lemmas}
\begin{lem}\label{H1}
Denote $(n^*_i)_{i\geq1}$ the decreasing rearrangement of
\begin{equation*}
\{|n|:\ \mbox{where $n$ is repeated}\ 2a_n+k_n+k_n'\ \mbox{times}\},
\end{equation*}
Then for any $\theta \in (0,1)$, one has
\begin{eqnarray}\label{H2}
\sum_{n\in\mathbb{Z}}(2a_n+k_n+k_n')|n|^{\theta}-2(n_1^*)^{\theta}++m^*{(k,k')}^{\theta}\geq(2-2^{\theta})\left(\sum_{i\geq 3}(n_i^*)^{\theta}\right).
\end{eqnarray}
\end{lem}
\begin{proof}  Without loss of generality,\ denote $(n_i)_{i\geq 1},\ |n_1|\geq |n_2|\geq\cdots$, the system $\{\mbox{$ n $ is repeated}\ 2a_n+k_n+k_n'\ \mbox{times}\}$ and we have $n_i^*=|n_i|\ \mbox{for}\ \forall\ i\geq1$. There exists $(\mu_i)_{i\geq1}$ with $\mu_i\in\{-1,1\}$ such that
\begin{equation*}
m(k,k')=\sum_{i\geq1}\mu_in_i,
\end{equation*}
and hence
\begin{equation*}
n_1^*\leq \sum_{i\geq 2}n_i^*+m^*(k,k').
\end{equation*}
Consequently
\begin{equation*}
(n_1^*)^{\theta}\leq\left(\sum_{i\geq 2}n_i^*+m^*(k,k')\right)^{\theta}.
\end{equation*}
Thus the inequality (\ref{H2}) will follow from the inequality
\begin{equation}\label{0001}
\sum_{i\geq 2}(n_i^*)^{\theta}+m^*(k,k')^{\theta}\geq \left(\sum_{i\geq 2}n_i^*+m^*(k,k')\right)^{\theta}+(2-2^{\theta})\left(\sum_{i\geq 3}(n_i^*)^\theta\right).
\end{equation}
To prove the inequality (\ref{0001}), one just needs the following fact: consider the function
\begin{equation*}
f(x)=(1+x)^{\theta}-x^{\theta},\qquad x\in[1,\infty),
\end{equation*}
and one has
\begin{equation}\label{0002}
\max_{x\in[1,\infty)}f(x)=f(1)=2^{\theta}-1,
\end{equation}
which is based on
\begin{equation*}
f'(x)=\theta((1+x)^{\theta-1}-x^{\theta-1})<0,\  \mbox{for}\ x\in[1,\infty)\ \mbox{and}\ \forall\  \theta\in(0,1).
\end{equation*}
Hence, for any $a\geq b>0$, we have
\begin{eqnarray}
\nonumber&&\left(a+b\right)^{\theta}+(2-2^{\theta})b^{\theta}-a^{\theta}-b^{\theta}\\
&=&\nonumber\left(a+b\right)^{\theta}-a^{\theta}+(1-2^{\theta})b^{\theta}\\
&=&\nonumber b^{\theta}\left(\left(1+\frac{a}{b}\right)^{\theta}-\left(\frac{a}{b}\right)^{\theta}-(2^{\theta}-1)\right)\\
\nonumber&\leq&0,
\end{eqnarray}
where the last inequality is based on (\ref{0002}).
That is
\begin{equation}\label{0003}
a^{\theta}+b^{\theta}\geq\left(a+b\right)^{\theta}+(2-2^{\theta})b^{\theta}.
\end{equation}
By iteration and in view of (\ref{0003}), one obtains
\begin{eqnarray*}
&&\sum_{i\geq 2}(n_i^*)^{\theta}+m^*(k,k')^{\theta}\\
%&=&|n_2|^{\theta}+|n_3|^{\theta}+\sum_{i\geq 4}|n_i|^{\theta}+m^*(k,k')^{\theta}\\
%&\geq&\left(|n_2|+|n_3|\right)^{\theta}+(2-2^{\theta})|n_3|^{\theta}+\sum_{i\geq 4}|n_i|^{\theta}\qquad (\mbox{in view of (\ref{0003})})\\
%&=&\left(|n_2|+|n_3|\right)^{\theta}+|n_4|^{\theta}+\sum_{i\geq 5}|n_i|^{\theta}+m^*(k,k')^{\theta}+(2-2^{\theta})|n_3|^{\theta}\\
%&\geq&\left(|n_2|+|n_3|+|n_4|\right)^{\theta}+(2-2^{\theta})|n_4|^{\theta}+\sum_{i\geq 5}|n_i|^{\theta}+(2-2^{\theta})|n_3|^{\theta}\\
%&&\qquad (\mbox{in view of (\ref{0003}) again})\\
%&=&\left(|n_2|+|n_3|+|n_4|\right)^{\theta}+\sum_{i\geq 5}|n_i|^{\theta}+m^*(k,k')^{\theta}+(2-2^{\theta})(|n_3|^{\theta}+|n_4|^{\theta})\\
%&&\cdots\\
&\geq&\left(\sum_{i\geq 2}n_i^*\right)^{\theta}+m^*(k,k')^{\theta}+(2-2^{\theta})\left(\sum_{i\geq 3}(n_i^*)^{\theta}\right)\\
&\geq&\left(\sum_{i\geq 2}n_i^*+m^*(k,k')\right)^{\theta}+(2-2^{\theta})\left(\sum_{i\geq 3}(n_i^*)^{\theta}\right),
\end{eqnarray*}
where the last inequality is based on
\begin{equation*}
|a|^{\theta}+|b|^{\theta}\geq \left({|a|+|b|}\right)^{\theta}
\end{equation*}
for all $a,b$ and $0<\theta<1$.
\end{proof}

\begin{lem}\label{a1}Let $\theta\in(0,1)$ and $k_n,k'_n\in\mathbb{N},|{\widetilde{V}}_n|\leq2\ \mbox{for}\ \forall \ n\in\mathbb{Z}$.
Assume further
\begin{equation}
\label{0004}\left|\sum_{n\in\mathbb{Z}}(k_n-k_n')(n^2+ \widetilde{V}_n)\right|\leq1.
\end{equation}
Then one has
\begin{equation}\label{0006}
\sum_{n\in\mathbb{Z}}|k_n-k_n'||n|^{\theta/2}\leq3\cdot 8^{\theta/2}\left(\sum_{i\geq3}|n_i|^{\theta}+m^{*}(k,k')^{\theta}\right),
\end{equation}
where $(n_i)_{i\geq1}, |n_1|\geq|n_2|\geq|n_3|\geq\cdots$, denote the system \{$n$: $n$ is repeated $k_n+k'_n$ times\}.
\end{lem}
\begin{proof}
From the definition of $(n_i)_{i\geq1}$, there exist $(\mu_i)_{i\geq1}$ with $\mu_i\in\{-1,1\}$ such that
\begin{equation}
\label{0007}m(k,k')=\sum_{i\geq1}\mu_in_i,
\end{equation}
and
\begin{equation}
\label{0008}\sum_{n\in\mathbb{Z}}(k_n-k_n')n^2=\sum_{i\geq1}\mu_in_i^2.
\end{equation}
In view of (\ref{0004}), (\ref{0008}) and $|\widetilde{V}_n|\leq 2$,
one has
\begin{equation*}
\left|\sum_{i\geq1}\mu_in_i^2\right|\leq\left|\sum_{n\in\mathbb{Z}}(k_n-k_n') \widetilde{V}_n\right|+1\leq2\sum_{n\in\mathbb{Z}}(k_n+k_n')+1,
\end{equation*}which implies
\begin{equation}\label{0009}
\left|n_1^2+\left(\frac{\mu_2}{\mu_1}\right)n_2^2\right|\leq2\sum_{i\geq1}1+\sum_{i\geq3}n_i^2+1\leq \sum_{i\geq 3}(2+n_i^2)+5.
\end{equation}
On the other hand,\ by (\ref{0007}),\ we obtain
\begin{equation}\label{0010}
\left|n_1+\left(\frac{\mu_2}{\mu_1}\right)n_2\right|\leq \sum_{i\geq 3}|n_i|+m^*(k,k').
\end{equation}
To prove the inequality (\ref{0006}),\ we will distinguish two cases:

\textbf{Case. 1.} $\frac{\mu_2}{\mu_1}=-1$.

\textbf{Case. 1.1.} $n_1=n_2$.

Then it is easy to show that
\begin{equation*}
\sum_{n\in\mathbb{Z}}|k_n-k_n'||n|^{\theta/2}\leq \sum_{i\geq3}|n_i|^{\theta/2}\leq3\cdot 8^{\theta/2}\left(\sum_{i\geq3}|n_i|^{\theta}+m^*(k,k')^{\theta}\right).
\end{equation*}

\textbf{Case. 1.2.} $n_1\neq n_2$.

Then one has
\begin{eqnarray}
\nonumber&&|n_1-n_2|+|n_1+n_2|\\
&\leq&\nonumber|n_1-n_2|+|n_1^2-n_2^2|\\
& \leq&\nonumber\sum_{i\geq 3}|n_i|+m^*(k,k')+\sum_{i\geq 3}(2+n_i^2)+5\quad \mbox{(in view of (\ref{0009}) and (\ref{0010}))}\\
\label{0011} &\leq&8\left(\sum_{i\geq3}|n_i|^2+m^{*}(k,k')^2\right).
\end{eqnarray}
Hence
\begin{equation*}
\max\{|n_1|,|n_2|\}\leq \max\{|n_1-n_2|,|n_1+n_2|\}\leq 8\left(\sum_{i\geq3}|n_i|^2+m^{*}(k,k')^2\right).
\end{equation*}
For $j=1,2,$ one has
\begin{equation*}
|n_j|^{\theta/2}\leq 8^{\theta/2}\left(\sum_{i\geq3}|n_i|^{2}+m^{*}(k,k')^2\right)^{\theta/2}\leq  8^{\theta/2}\left(\sum_{i\geq3}|n_i|^{\theta}+m^{*}(k,k')^{\theta}\right),
\end{equation*}
where the last inequality is based on
the fact that the function $|x|^{\theta/2}$ is a concave function for $0<\theta<1$. Therefore,
\begin{equation}\label{0012}
|n_1|^{\theta/2}+|n_2|^{\theta/2}\leq 2 \cdot 8^{\theta/2}\left(\sum_{i\geq3}|n_i|^{\theta}+m^{*}(k,k')^{\theta}\right).
\end{equation}
Now one has
\begin{eqnarray}
\nonumber\sum_{n\in\mathbb{Z}}|k_n-k_n'||n|^{\theta/2}
&\leq&\nonumber\sum_{n\in\mathbb{Z}}(k_n+k_n')|n|^{\theta/2}\\
&=&\nonumber\sum_{i\geq1}|n_i|^{\theta/2}\\
&\leq&\nonumber\left(|n_1|^{\theta/2}+|n_2|^{\theta/2}\right)+\sum_{i\geq3}|n_i|^{\theta}\\
&\leq&\nonumber(2 \cdot 8^{\theta/2}+1)\left(\sum_{i\geq3}|n_i|^{\theta}+m^{*}(k,k')^{\theta}\right)\ \ \mbox{(in view of (\ref{0012}))}\\
&\leq&\label{0013}3 \cdot 8^{\theta/2}\left(\sum_{i\geq3}|n_i|^{\theta}+m^{*}(k,k')^{\theta}\right).
\end{eqnarray}
\textbf{Case. 2.} $\frac{\mu_2}{\mu_1}=1$.

In view of (\ref{0009}), one has
\begin{equation*}
n_1^2+n_2^2\leq 7\sum_{i\geq3}|n_i|^{2},
\end{equation*}which implies
\begin{equation*}
|n_j|^{\theta/2}\leq 7^{\theta/2}\left(\sum_{i\geq3}|n_i|^{2}\right)^{\theta/2}\leq  7^{\theta/2}\sum_{i\geq3}|n_i|^{\theta}\ \ \mbox{($j=1,2$)}.
\end{equation*}
Therefore,
\begin{equation}\label{0014}
|n_1|^{\theta/2}+|n_2|^{\theta/2}\leq 2 \cdot 7^{\theta/2}\sum_{i\geq3}|n_i|^{\theta}.
\end{equation}
Following the proof of (\ref{0013}), we have
\begin{equation*}
\sum_{n\in\mathbb{Z}}|k_n-k_n'||n|^{\theta/2}\leq
3 \cdot 8^{\theta/2}\left(\sum_{i\geq3}|n_i|^{\theta}+m^{*}(k,k')^{\theta}\right).
\end{equation*}

\end{proof}

\subsection{Proof of Lemma \ref{N2}}
\begin{proof}
Firstly, we will prove the inequality (\ref{053}).
Write $\mathcal{M}_{akk'}$ in the form of
\begin{equation*}
\mathcal{M}_{akk'}=\mathcal{M}_{abll'}=\prod_nI_n(0)^{a_n}I_n^{b_n}q_n^{l_n}{\bar q_n}^{l_n'},
\end{equation*}
where
\begin{equation*}
b_n=k_n\wedge k_n',\quad l_n=k_n-b_n,\quad l_n'=k_n'-b_n'
\end{equation*}
and
$l_nl_n'=0$ for all $n$.

Express the term
\begin{equation*}
\prod_nI_n^{b_n}=\prod_n(I_n(0)+J_n)^{b_n}
\end{equation*}by the monomials of the form
\begin{equation*}
\prod_nI_n(0)^{b_n},
\end{equation*}
\begin{equation*}
\sum_{m,b_m\geq 1}\left(I_m(0)^{b_m-1}J_m\right)\left(\prod_{n\neq m}I_n(0)^{b_n}\right),
\end{equation*}
\begin{equation*}
\sum_{m,b_m\geq2\atop
r\leq b_m-2}\left(\prod_{n< m}I_n(0)^{b_n}\right)\left(I_m(0)^{r}J_m^2I_m^{b_m-r-2}\right)\left(\prod_{n> m}I_n^{b_n}\right),
\end{equation*}
and
\begin{eqnarray*}
&&\sum_{m_1< m_2,b_{m_1},b_{m_2}\geq 1\atop
r\leq b_{m_2}-1}\left(\prod_{n< m_1}I_n(0)^{b_n}\right)\left(I_{m_1}(0)^{b_{m_1}-1}J_{m_1}\right)
\\
&&\nonumber\times\left(\prod_{m_1<n< m_2}I_n(0)^{b_n}\right)\left(I_{m_2}(0)^{r}J_{m_2}I_{m_2}^{b_{m_2}-r-1}\right)
\left(\prod_{n> m_2}I_n^{b_n}\right).
\end{eqnarray*}

Now we will estimate the bounds for the coefficients respectively.

Consider the term
$\mathcal{M}_{akk'}=\prod_nI_n(0)^{a_n}q_n^{k_n}\bar q_n^{k_n'}$ with fixed $a,k,k'$ satisfying $k_nk_n'=0$ for all $n$. It is easy to see that $\mathcal{M}_{akk'}$ comes from some parts of the terms $\mathcal{M}_{\alpha\kappa\kappa'}$ with no assumption for $\kappa$ and $\kappa'$. For any given $n$ one has
\begin{equation*}
I_n(0)^{a_n}q_n^{k_n}\bar q_n^{k_n'}=\sum_{\beta_n=k_n\wedge k_n'}I_n(0)^{\alpha_n+\beta_n}q_n^{\kappa_n-\beta_n}\bar q_n^{\kappa_n'-\beta_n}.
\end{equation*}
Hence,
\begin{equation}\label{055}
\alpha_n+\beta_n=a_n,
\end{equation}
and
\begin{equation}\label{056}
\kappa_n-\beta_n=k_n,\qquad \kappa_n'-\beta_n=k_n'.
\end{equation}
Therefore, if $0\leq\alpha_n\leq a_n$ is chosen, so $\beta_n,k_n,k_n'$ are determined.
On the other hand,
\begin{eqnarray}
\nonumber|B_{\alpha\kappa\kappa'}|
&\leq&\nonumber ||R||_{\rho,\mu}e^{\rho\left(\sum_{n}(2\alpha_n+\kappa_n+\kappa_n')|n|^{\theta}-2(n_1^*)^{\theta}\right)-\mu m^{\ast}(\kappa,\kappa')^{\theta}}\qquad\qquad\qquad\qquad\ \\&&\nonumber{(\mbox{in view of (\ref{H00})})}\\
&=&\nonumber||R||_{\rho,\mu}e^{\rho\left(\sum_{n}(2\alpha_n+(k_n+a_n-\alpha_n)+(k_n'+a_n-\alpha_n))|n|^{\theta}-2(n_1^*)^{\theta}\right)-\mu m^{\ast}(\kappa,\kappa')^{\theta}}\quad\\
&&\nonumber{(\mbox{in view of (\ref{055}) and (\ref{056})})}\\
&=&\nonumber||R||_{\rho,\mu}e^{\rho\left(\sum_{n}(2a_n+k_n+k_n')|n|^{\theta}-2(n_1^*)^{\theta}\right)-\mu m^*(k,k')^{\theta}}.
\end{eqnarray}
Hence,
\begin{equation}\label{057}
|B_{akk'}|\leq||R||_{\rho,\mu}\prod_n(1+a_n)e^{\rho\left(\sum_{n}(2a_n+k_n+k_n')|n|^{\theta}-2(n_1^*)^{\theta}\right)-\mu m^*(k,k')^{\theta}}.
\end{equation}
Similarly,
\begin{eqnarray*}
\left|B_{akk'}^{(m)}\right|&\leq& ||R||_{\rho,\mu}\left(\prod_{n\neq m}(1+a_n)\right)(1+a_m)^2e^{\rho\left(\sum_{n}(2a_n+k_n+k_n')|n|^{\theta}+2|m|^{\theta}-2(n_1^*)^{\theta}\right)-\mu m^*(k,k')^{\theta}},\\
\left|B_{akk'}^{(m,m)}\right|&\leq& ||R||_{\rho,\mu}\left(\prod_{n\neq m}(1+a_n)\right)(1+a_m)^3e^{\rho\left(\sum_{n}(2a_n+k_n+k_n')|n|^{\theta}
+4|m|^{\theta}-2(n_1^*)^{\theta}\right)-\mu m^*(k,k')^{\theta}},\\
\left|B_{akk'}^{(m_1,m_2)}\right|&\leq& ||R||_{\rho,\mu}\left(\prod_{n<m_1}(1+a_n)\right)(1+a_{m_1})^2\left(\prod_{m_1<n<m_2 }(1+a_n)\right)\\&&\times (1+a_{m_2})^2e^{\rho\left(\sum_{n}(2a_n+k_n+k_n')|n|^{\theta}
+2|m_1|^{\theta}+2|m_2|^{\theta}-2(n_1^*)^{\theta}\right)-\mu m^*(k,k')^{\theta}}.
\end{eqnarray*}
In view of (\ref{050}) and (\ref{057}), we have
\begin{eqnarray}\label{058}
||R_0||_{\rho+\delta,\mu-\delta}^{+}
\leq||R||_{\rho,\mu}\prod_n(1+a_n)e^{-\delta\left(\sum_{n}(2a_n+k_n+k_n')|n|^{\theta}
-2(n_1^*)^{\theta}+m^*(k,k')^{\theta}\right)}.
\end{eqnarray}
Now we will show that
\begin{eqnarray}
\label{059}\prod_{n}(1+a_n)e^{-\delta\left(\sum_{n}(2a_n+k_n+k_n')|n|^{\theta}
-2(n_1^*)^{\theta}+m^*(k,k')^{\theta}\right)}&\leq& \left(\frac{1}{\delta}\right)^{ C(\theta)\delta^{-\frac{1}{\theta}}}.
\end{eqnarray}
%where $C(\theta)$ is a positive constant depending only on $\theta$.

\textbf{Case 1.} $n_1^*=n_2^*=n_3^*.$
Then one has
\begin{eqnarray*}
(\ref{059})
&=&\nonumber \prod_{n}(1+a_n)e^{-{\delta}\sum_{i\geq3}|n_i|^{\theta}}e^{-\delta m^*(k,k')^{\theta}}\\
&\leq&\nonumber\prod_n(1+a_n)e^{-\frac{\delta}{3}\sum_{i\geq1}|n_i|^{\theta}}\\
&=&\nonumber\prod_n(1+a_n)e^{-\frac\delta3\sum_{n}(2a_n+k_n+k_n')|n|^{\theta}}\\
&\leq&\nonumber\prod_n\left((1+a_n)e^{-\frac{2\delta}{3}a_n|n|^\theta}\right)\\
&\leq&\left(\frac{1}{\delta}\right)^{C(\theta){\delta}^{-\frac{1}{\theta}}}\ \ \mbox{(in view of Lemma 7.6 in \cite{CLSY})}.
\end{eqnarray*}

\textbf{Case 2.} $n_1^*>n_2^*=n_3^*.$ In this case, $a_{n}=1$ for $n=n_1$.
Then we have
\begin{eqnarray*}
(\ref{059})
&=&2\left(\prod_{|n|\leq n_2^*}(1+a_n)e^{-(2-2^\theta)\delta\sum_{i\geq3}(n_i^*)^{\theta}}\right)
\\
&\leq&2\prod_{|n|\leq n_2^*}(1+a_n)e^{-\frac12(2-2^\theta)\delta\sum_{i\geq2}(n_i^*)^{\theta}}\\
&=&\nonumber\prod_{|n|\leq n_2^*}(1+a_n)e^{-\frac12(2-2^\theta)\delta\sum_{|n|\leq n_2^*}(2a_n+k_n+k_n')|n|^{\theta}}\\
&\leq&2\prod_{|n|\leq n_2^*}\left( (1+a_n)e^{-(2-2^\theta)\delta a_n|n|^\theta}\right)\\
&\leq&\left(\frac{1}{\delta}\right)^{C(\theta){\delta}^{-\frac{1}{\theta}}}\ \ \mbox{(in view of Lemma 7.6 in \cite{CLSY})}.
\end{eqnarray*}

\textbf{Case 3.} $n_{1}^{\ast}\geq n_2^*>n_3^*.$ In this case, $a_{n}=1$ or $2$ for $n\in\{n_1, n_2\}$.
Hence
\begin{eqnarray*}
(\ref{059})
&\leq&4\left(\prod_{|n|\leq n_3^*}(1+a_n)e^{-\delta\sum_{i\geq3}(n_i^*)^{\theta}}\right)
\\
%&&\mbox{(in view of (\ref{}))}\\
&\leq&\nonumber4\prod_{|n|\leq n_3^*}(1+a_n)e^{-\delta\sum_{|n|\leq n_3^*}(2a_n+k_n+k_n')|n|^{\theta}}\\
&\leq&4\prod_{|n|\leq n_3^*}\left( (1+a_n)e^{-2\delta a_n|n|^\theta}\right)\\
&\leq&\left(\frac{1}{\delta}\right)^{C(\theta){\delta}^{-\frac{1}{\theta}}}\ \ \mbox{(in view of Lemma 7.6 in \cite{CLSY})}.
\end{eqnarray*}
We finished the proof of ({\ref{059}}).

Similarly, one has
\begin{eqnarray*}
&&||R_i||_{\rho+\delta,\mu-\delta}^{+}\leq\left(\frac{1}{\delta}\right)^{C(\theta)\delta^{-\frac{1}{\theta}}}||R_i||_{\rho,\mu},\qquad i=1,2,
\end{eqnarray*}
and hence
\begin{equation*}
||R||_{\rho+\delta,\mu-\delta}^{+}\leq\left(\frac{1}{\delta}\right)^{ C(\theta)\delta^{-\frac{1}{\theta}}}||R||_{\rho,\mu}.
\end{equation*}
On the other hand, the coefficient of $\mathcal{M}_{abll'}$ increases by at most a factor $$\left(\sum_{n}(a_n+b_n)\right)^2,$$ then
\begin{eqnarray}
\nonumber||R||_{\rho+\delta,\mu-\delta}
&\leq&\nonumber||R||_{\rho,\mu}^{+}\left(\sum_{n}(a_n+b_n)\right)^2
e^{-\delta(\sum_{n}(2a_n+k_n+k_n')|n|^{\theta}-2(n_1^*)^{\theta}+m^*(\kappa,\kappa)^{\theta})}\\
\nonumber&\leq&||R||_{\rho,\mu}^{+}\left(2\sum_{i\geq 3}(n_i^*)^{\theta}\right)^2 e^{-\delta(2-2^\theta)\sum_{i\geq3}(n_i^*)^{\theta}}\quad (\mbox{in view of (\ref{016})})\\
\label{060} &\leq&\frac{4}{(2-2^{\theta})^2\delta^2}||R||_{\rho,\mu}^{+},
\end{eqnarray}
where the last inequality is based on Lemma 7.5 in \cite{CLSY} with $p=2$.
\end{proof}
\subsection{Proof of Lemma \ref{N3}}
\begin{proof}

We distinguish two cases:

$\textbf{Case. 1.}$ $$\left|\sum_{n\in\mathbb{Z}}(k_n-k'_n)n^2\right|>10\sum_{n\in\mathbb{Z}}|k_n-k'_n|.$$\\
Since $|\widetilde{{V}}_n|\leq2$, we have
$$\left|\sum_{n\in\mathbb{Z}}(k_n-k'_n)(n^2+\widetilde{V}_n)\right|>10\sum_{n\in\mathbb{Z}}|k_n-k'_n|-2\sum_{n\in\mathbb{Z}}|k_n-k'_n|\geq1,$$
where the last inequality is based on $\mbox{supp}\ k\bigcap \mbox{supp}\ k'=\emptyset$.
There is no small divisor and (\ref{061}) holds trivially.

$\textbf{Case. 2.}$ $$\left|\sum_{n\in\mathbb{Z}}(k_n-k'_n)n^2\right|\leq 10\sum_{n\in\mathbb{Z}}|k_n-k'_n|.$$
In this case, we always assume  $$\left|\sum_{n\in\mathbb{Z}}(k_n-k_n')(n^2+ \widetilde{V}_n)\right|\leq1,$$otherwise there is no small divisor.

Firstly, one has
\begin{eqnarray}
\nonumber&&\sum_{n\in\mathbb{Z}}|k_n-k_n'||n|^{\theta/2}\\
\nonumber&\leq&\nonumber{3\cdot 8^{\theta/2}}\left(\sum_{i\geq3}(n_i^*)^{\theta}+m^{*}(k,k')^{\theta}\right)\qquad (\mbox{in view of Lemma \ref{a1}})\\
\label{062}&\leq&\frac{3\cdot 8^{\theta/2}}{2-2^{\theta}}\left(\sum_{n\in\mathbb{Z}}(2a_n+k_n+k_n')|n|^{\theta}-2(n_1^*)^{\theta}+2m^*(k,k')^{\theta}\right),
\end{eqnarray}
where the last inequality is based on Lemma \ref{H1}.

Since $\sum_{n\in\mathbb{Z}}(k_n-k'_n)n^2\in\mathbb{Z},$
the Diophantine property of $(\widetilde{V}_n)$ implies
\begin{equation}\label{063}
\left|\sum_{n\in\mathbb{Z}}(k_n-k'_n)(n^2+\widetilde{V}_n)\right|\geq\frac\gamma2\prod_{n\in\mathbb{Z}}\frac{1}{1+{|k_n-k'_n|}^2|n|^4}.
\end{equation}
Hence,
\begin{eqnarray}
\nonumber&&{|{F}_{akk'}|}e^{-(\rho+\delta)(\sum_{n}(2a_n+k_n+k'_n)|n|^{\theta}-2(n_1^{*})^{\theta})+(\mu-2\delta)m^{*}(k,k')^{\theta}}\\
&=&\nonumber\frac{|{B}_{akk'}|}{|\sum_{n}(k_n-k'_n)(n^2+\widetilde{V}_n)|}\\
\nonumber&&\times e^{-(\rho+\delta)(\sum_{n}(2a_n+k_n+k'_n)|n|^{\theta}-2(n_1^{*})^{\theta})+(\mu-2\delta)m^{*}(k,k')^{\theta}}\\
&&\nonumber\mbox{(in view of (\ref{044}))}\\
&=&\nonumber|{B}_{akk'}|e^{-\rho\left(\sum_{n}(2a_n+k_n+k'_n)|n|^{\theta}-2(n_1^{*})^{\theta}\right)+\mu m^*(k,k')^{\theta}}
\\
&&\nonumber\times \frac{e^{-\delta\left(\sum_{n}(2a_n+k_n+k'_n)|n|^{\theta}-2(n_1^{*})^{\theta}+2m^*(k,k')^{\theta}\right)}}
{|\sum_{n}(k_n-k'_n)(n^2+\widetilde{V}_n)|}\\
\nonumber&\leq& 2\gamma^{-1}||{R_0}||_{\rho,\mu}^{+} \prod_{n}\left({1+{|k_n-k'_n|}^2|n|^4}\right)\\
&&\nonumber\times e^{-\delta\left(\sum_{n}(2a_n+k_n+k'_n)|n|^{\theta}-2(n_1^*)^{\theta}+2m^*(k,k')^{\theta}\right)}\\
\nonumber&&
 \mbox{(in view of (\ref{050}) and (\ref{063}))}\\
\nonumber&\leq& 2\gamma^{-1}||{R_0}||_{\rho,\mu}^{+} \prod_{n}\left({1+{|k_n-k'_n|}^2|n|^4}\right)\\
\nonumber &\leq&2\gamma^{-1} ||{R_0}||_{\rho,\mu}^+e^{\sum_{n}\ln(1+|k_n-k'_n|^2|n|^4)}e^{-\frac{2-2^{\theta}}
{3\cdot 8^{\theta/2}} \delta\sum_n\left(|k_n-k_n'||n|^{\theta/2}\right)}\\
&&\nonumber\mbox{(in view of (\ref{062}))}\\
\nonumber&=& 2\gamma^{-1} ||{R_0}||_{\rho,\mu}^+e^{\sum_{n}\ln(1+|k_n-k'_n|^2|n|^4)}e^{-\tilde{\delta} \sum_n\left(|k_n-k_n'||n|^{\theta/2}\right)}\\
&&\nonumber \mbox{ (note $\tilde{\delta}=\frac{2-2^{\theta}}{3\cdot 8^{\theta/2}} \delta$) }\\
&=&\nonumber2\gamma^{-1} ||{R_0}||_{\rho,\mu}^+e^{\sum_{{n:k_n\neq k'_n}}\ln(1+|k_n-k'_n|^2|n|^4)-\tilde{\delta}\sum_{n:k_n\neq k'_n}\left(|k_n-k_n'||n|^{\theta/2}\right)}\\
&\leq&\nonumber2\gamma^{-1} ||{R_0}||_{\rho,\mu}^+e^{8\left(\sum_{{n:k_n\neq k'_n}}\ln(|k_n-k'_n||n|)\right)+3-\tilde{\delta}\sum_{n:k_n\neq k'_n}\left(|k_n-k_n'|^{\theta/2}|n|^{\theta/2}\right)}\\
&&\mbox{(in view of $0<\theta<1$)}\nonumber\\
&=&\nonumber \frac{2e^{3}}{\gamma} ||{R_0}||_{\rho,\mu}^+e^{\sum_{n:k_n\neq k'_n}\left(8\ln(|k_n-k'_n||n|)-\tilde{\delta}|k_n-k_n'|^{\theta/2}|n|^{\theta/2}\right)}\\
&=&\nonumber\frac{2e^{3}}{\gamma} ||{R_0}||_{\rho,\mu}^+e^{\sum_{|n|\leq N:k_n\neq k'_n}\left(8\ln(|k_n-k'_n||n|)-\tilde{\delta}|k_n-k_n'|^{\theta/2}|n|^{\theta/2}\right)}\\
&&+\nonumber\frac{2e^{3}}{\gamma} ||{R_0}||_{\rho,\mu}^+e^{\sum_{n>N:k_n\neq k'_n}\left(8\ln(|k_n-k'_n||n|)-\tilde{\delta}|k_n-k_n'|^{\theta/2}|n|^{\theta/2}\right)}\\
&&\nonumber \mbox{(where $N=\left(\frac{16}{\theta \tilde{\delta} }\right)^{4/\theta}$)}\\
&=&\nonumber\frac{2e^{3}}{\gamma} ||{R_0}||_{\rho,\mu}^+e^{\left(\frac{16}{\theta \tilde{\delta}}\right)^{4/\theta}\cdot \frac{32}{\theta}\ln\left(\frac{16}{\theta \tilde{\delta}}\right)}\qquad{\mbox{(in view of  (\ref{065}) below)}}\\
&&+\nonumber\frac{2e^{3}}{\gamma} ||{R_0}||_{\rho,\mu}^+\qquad\qquad\qquad\qquad\mbox{(in view of (\ref{066}) below)}\\
&\leq&\label{064}\frac{1}{\gamma}\cdot e^{C(\theta)\delta^{-\frac5\theta}} ||{R_0}||_{\rho,\mu}^+ \  \ \mbox {(for $0<\delta\ll1$)},
\end{eqnarray}
where $C(\theta)$ is a positive constant depending on $\theta$ only.\\
Therefore, in view of (\ref{050}) and (\ref{064}), we finish the proof of (\ref{061}) for $i=0$.

It is easy to verify the following two facts that
\begin{equation}\label{065}
\max_{x\geq 1} f(x)=f\left(\left(\frac{16}{\theta\delta}\right)^{2/\theta}\right)=-\frac {16}{\theta} +8\ln\left(\left(\frac{16}{\theta\delta}\right)^{2/\theta}\right)\leq \frac{16}{\theta}\ln\left(\frac{16}{\theta\delta}\right)
\end{equation}
with $f(x)=(-\delta x^{\theta/2}+8\ln x)$,
and when $|n|>N=\left(\frac{16}{\theta\delta}\right)^{4/\theta}, k_n\neq k'_n$, one has
\begin{equation}\label{066}
-\delta\left(|k_n-k_n'|^{\theta/2}|n|^{\theta/2}\right)+8\ln(|k_n-k'_n||n|)<0\ \  \mbox {(for $0<\delta\ll1$)}.
\end{equation}
Similarly, one can prove (\ref{061}) for $i=1$.
\end{proof}
\subsection{Proof of Lemma \ref{H3}}
\begin{proof}
Let
\begin{equation*}
H_1=\sum_{a,k,k'} b_{akk'}\mathcal{M}_{akk'}
\end{equation*}
and
\begin{equation*}
H_2=\sum_{A,K,K'} B_{AKK'}\mathcal{M}_{AKK'}.
\end{equation*}
It follows easily that
\begin{equation*}
\{H_1,H_2\}=\sum_{a,k,k',A,K,K'}b_{akk'}B_{AKK'}\{\mathcal{M}_{akk'},\mathcal{M}_{AKK'}\},
\end{equation*}
where
\begin{eqnarray*}
\{\mathcal{M}_{akk'},\mathcal{M}_{AKK'}\}
&=&\frac{1}{2\textbf{i}}\sum_j\left(\prod_{n\neq j}I_n(0)^{a_n+A_n}q_n^{k_n+K_n}\bar{q}_n^{k_n'+K_n'}\right)\\
&&\times\left((k_jK_j'-k_j'K_j)I_j(0)^{a_j+A_j}q_j^{k_j+K_j-1}\bar{q}_j^{k_j'+K_j'-1}\right).
\end{eqnarray*}
Then the coefficient of
\[
\mathcal{M}_{\alpha\kappa\kappa'} :=\prod_{n}I_n(0)^{\alpha_n}q_n^{\kappa_n}\bar{q}_n^{\kappa'_n}
\]
is given by
\begin{equation}\label{006*}
B_{\alpha\kappa\kappa'}= \frac{1}{2\textbf{i}}\sum_{j}\sum_{*}\sum_{**}(k_jK_j'-k_j'K_j)b_{akk'}B_{AKK'},
\end{equation}
where
\begin{equation*}
\sum_{*}=\sum_{a,A \atop a+A=\alpha},
\end{equation*}
and
\begin{equation*}
\sum_{**}=\sum_{k,k',K,K'\atop \mbox{when}\ n\neq j, k_n+K_n=\kappa_n,k_n'+K_n'=\kappa_n';\mbox{when}\ n=j, k_n+K_n-1=\kappa_n,k_n'+K_n'-1=\kappa_n'}.
\end{equation*}
In view of (\ref{H00}) and Lemma \ref{H1}, one has
\begin{eqnarray}
\label{007*} |b_{akk'}|
  &\leq& ||H_1||_{\rho-\delta_{1},\mu+2\delta_{1}}
e^{\rho\sum_{n}(2a_n+k_n+k_n')|n|^{\theta}-2\rho(n_1^*)^{\theta}-\mu m^*(k,k')^{\theta}}\\
&&\nonumber
\times
e^{-(2-2^{\theta})\delta_{1}
\sum_{i\geq 3}(n_i^*)^{\theta}-\delta_{1} m^*(k,k')^{\theta}},
\end{eqnarray}
and
\begin{eqnarray}\label{008*}
|B_{AKK'}|&\leq&||H_2||_{\rho-\delta_{2},\mu+2\delta_{2}}e^{\rho\sum_{n}(2A_n+K_n+K_n')|n|^{\theta}-2\rho(N_1^*)^{\theta}
-\mu m^*(K,K')^{\theta}}\\
&&\nonumber
\times
e^{-(2-2^{\theta})\delta_{2}\sum_{i\geq 3}(N_i^*)^{\theta}-\delta_{2} m^*(K,K')^{\theta}}.
\end{eqnarray}
Substitution of (\ref{007*}) and (\ref{008*}) in (\ref{006*}) gives
\begin{eqnarray*}
\label{009*}|B_{\alpha\kappa\kappa'}|&\leq& \frac{1}{2} ||H_1||_{\rho-\delta_{1},\mu+2\delta_{1}}||H_2||_{\rho-\delta_{2},\mu+2\delta_{2}}
\times\sum_{j}\sum_{*}\sum_{**}|k_jK_j'-k_j'K_j|\\
\nonumber &&\times e^{\rho\left(\sum_{n}(2(a_n+A_{n})+k_n +K_{n}+k_n'+K_{n}')|n|^{\theta}-2(n_1^*)^{\theta}-2(N_1^*)^{\theta}\right)-\mu (m^{\ast}(k,k')^{\theta}+ m^*(K,K')^{\theta})}\\
\nonumber &&\times e^{-(2-2^{\theta})\left(\delta_{1}\sum_{i\geq 3}(n_i^*)^{\theta}+\delta_{2}\sum_{i\geq 3}(N_i^*)^{\theta}\right)}e^{-\left(\delta_{1}m^*(k,k')^{\theta}+\delta_{2}m^*(K,K')^{\theta}\right)}.
\end{eqnarray*}
Noting that
\begin{equation}\label{012}
\sum_n(2\alpha_n+\kappa_n+\kappa_n')=\sum_n(2a_n+k_n+k_n')+\sum_n(2A_n+K_n+K_n')-2
\end{equation}
and
\begin{eqnarray}
\nonumber&&\sum_n(2\alpha_n+\kappa_n+\kappa_n')|n|^{\theta}\\
 \label{013}&=&\sum_n(2a_n+k_n+k_n')|n|^{\theta}
+\sum_n(2A_n+K_n+K_n')|n|^{\theta}-2|j|^{\theta}.
\end{eqnarray}
Then one has
\begin{eqnarray*}
\nonumber |B_{\alpha\kappa\kappa'}|
 &\leq& ||H_1||_{\rho-\delta_{1},\mu+2\delta_{1}}||H_2||_{\rho-\delta_{2},\mu+2\delta_{2}}
e^{\rho(\sum_{n}(2\alpha_n+\kappa_n+\kappa_n')|n|^{\theta}-2(\nu_1^*)^{\theta})-\mu m^{\ast}(\kappa,\kappa')^{\theta}}\\
\nonumber&& \times\frac{1}{2}\sum_{j}\sum_{*}\sum_{**}|k_jK_j'-k_j'K_j| e^{2\rho(|j|^{\theta}+(\nu_1^*)^{\theta}-(n_1^*)^{\theta}-(N_1^*)^{\theta})}\\
&&\nonumber\times e^{\mu \left(m^*(\kappa,\kappa')^{\theta}- m^*(k,k')^{\theta}-m^*(K,K')^{\theta}\right)}\\
\nonumber &&\times e^{-(2-2^{\theta})\left(\delta_{1}\sum_{i\geq 3}(n_i^*)^{\theta}+\delta_{2}\sum_{i\geq 3}(N_i^*)^{\theta}\right)}e^{-\left(\delta_{1}m^*(k,k')^{\theta}+\delta_{2}m^*(K,K')^{\theta}\right)},
\end{eqnarray*}
where
\begin{equation*}
\nu_1^*=\max\{|n|:\alpha_n+\kappa_n+\kappa_n'\neq0\}.
\end{equation*}
To show (\ref{H4}) holds, it suffices to prove
\begin{equation}\label{H5}
I\leq \frac{1}{\delta_{2}}\left(\frac{1}{\delta_{1}}\right)^{C({\theta}){\delta_{1}^{-\frac{1}{\theta}}}},
\end{equation}
where
\begin{eqnarray*}
I&=& \frac{1}{2}\sum_{j}\sum_{*}\sum_{**}|k_jK_j'-k_j'K_j|
e^{2\rho(|j|^{\theta}+(\nu_1^*)^{\theta}-(n_1^*)^{\theta}-(N_1^*)^{\theta})}\\
\nonumber&&\nonumber\times e^{\mu (m^*(\kappa,\kappa')- m^*(k,k')- m^*(K,K'))}\\
&&\times e^{-(2-2^{\theta})\left(\delta_{1}\sum_{i\geq 3}(n_i^*)^{\theta}+\delta_{2}\sum_{i\geq 3}(N_i^*)^{\theta}\right)}e^{-\left(\delta_{1}m^*(k,k')^{\theta}+\delta_{2}m^*(K,K')^{\theta}\right)}.
\end{eqnarray*}
To this end, we first note some simple facts:

$\textbf{1}.$ If $j\notin\ \mbox{supp}\ (k+k') \bigcap\ \mbox{supp}\ (K+K')$, then
\begin{equation*}
\frac{\partial\mathcal{M}_{akk'}}{\partial q_j}
\frac{\partial\mathcal{M}_{AKK'}}{\partial \bar{q}_j}-\frac{\partial\mathcal{M}_{akk'}}{\partial \bar{q}_j}
\frac{\partial\mathcal{M}_{AKK'}}{\partial {q}_j}=0.
\end{equation*}
Hence we always assume $j\in\ \mbox{supp}\ (k+k') \bigcap\ \mbox{supp}\ (K+K')$. Therefore one has
\begin{equation*}
|j|\leq \min\{n_1^*,N^*_1\}.
\end{equation*}

$\textbf{2}.$ The following inequality always holds
\begin{equation}\label{014}
\nu_1^*\leq \max\{n_1^*,N_1^*\},
\end{equation}
and then one has
\begin{equation*}
|j|^{\theta}+(\nu_1^*)^{\theta}-(n_1^*)^{\theta}-(N_1^*)^{\theta}\leq 0.
\end{equation*}

$\textbf{3}.$ It is easy to see
\begin{eqnarray}
\nonumber\sum_{i\geq 1}(n_i^*)^{\theta}
\nonumber&=&\sum_n(2a_n+k_n+k_n')|n|^{\theta}\\
\nonumber&\geq&\sum_n(2a_n+k_n+k_n')\\
\label{015}&\geq&\sum_n(k_n+k_n')
\end{eqnarray}
and
\begin{eqnarray}
\nonumber\sum_{i\geq 3}(N_i^*)^{\theta}
\nonumber&\geq&\sum_n(2A_n+K_n+K_n')-2\\
\nonumber&\geq&\frac12\sum_n(2A_n+K_n+K_n')\\
\label{016}&\geq&\frac12\sum_{n}(K_n+K_n').
\end{eqnarray}
Based on (\ref{015}) and (\ref{016}), we obtain
\begin{eqnarray}
\sum_{n}(k_n+k_n')(K_n+K_n')\nonumber
&\leq& \left(\sup_{n}(K_n+K_n')\right)\left(\sum_{n}(k_n+k_n')\right)\\
&\leq&\label{017} 2\left(\sum_{i\geq 1}(n_i^*)^{\theta}\right)\left(\sum_{i\geq 3}(N_i^*)^{\theta}\right).
\end{eqnarray}
In view of (\ref{012}) and (\ref{016}),\ we have
\begin{eqnarray}\label{009**}
\sum_n(2a_n+\kappa_n+\kappa_n') \leq 2\left(\sum_{i\geq 1}(n_i^*)^{\theta}\right) +2\left(\sum_{i\geq 3}(N_i^*)^{\theta}\right)
\end{eqnarray}
$\textbf{4}.$ It is easy to see
\begin{equation*}
m(\kappa,\kappa')=m(k,k')+m(K,K').
\end{equation*}
Hence,
\begin{equation*}
m^*(\kappa,\kappa')\leq m^*(k,k')+m^*(K,K').
\end{equation*}
Moreover, one has
\begin{equation*}
m^*(\kappa,\kappa')^{\theta}\leq m^*(k,k')^{\theta}+m^*(K,K')^{\theta}.
\end{equation*}
which implies
\begin{equation}\label{018}
e^{\mu (m^*(\kappa,\kappa')^{\theta}- m^*(k,k')^{\theta}- m^*(K,K')^{\theta})}\leq 1.
\end{equation}
Now we will prove the inequality (\ref{H5}) holds:

$\textbf{Case. 1.} \ \nu_1^*\leq N_1^*$.

$\textbf{Case. 1.1.}\ |j|\leq n_3^*.$

Then one has
\begin{eqnarray}\label{019}
e^{2\rho(|j|^{\theta}-(n_1^*)^{\theta})}
&\leq& e^{(2-2^{\theta})\delta_{1}((n_3^*)^{\theta}-(n_1^*)^{\theta})},
\end{eqnarray}
if
\begin{eqnarray}\label{020}
\delta_{1}\leq \frac{2\rho}{2-2^{\theta}}.
\end{eqnarray}
Hence one obtains
\begin{eqnarray}
\nonumber &&e^{2\rho(|j|^{\theta}+(\nu_1^*)^{\theta}-(n_1^*)^{\theta}-(N_1^*)^{\theta})}
e^{-(2-2^{\theta})\delta_{1}\sum_{i\geq 3}(n_i^*)^{\theta}}\\
&\leq&\nonumber e^{(2-2^{\theta})\delta_{1}((n_3^*)^{\theta}-(n_1^*)^{\theta})}
e^{-(2-2^{\theta})\delta_{1}\sum_{i\geq 3}(n_i^*)^{\theta}}\\
&&\nonumber \mbox{(in view of $\nu_1^*\leq N_1^*$ and (\ref{019}))}\\
&=&\nonumber e^{-(2-2^{\theta})\delta_{1}((n_1^*)^{\theta}+\sum_{i\geq 4}(n_i^*)^{\theta})}\\
\label{021}&\leq& e^{-\frac{(2-2^{\theta})\delta_{1}}3\sum_{i\geq 1}(n_i^*)^{\theta}}.
\end{eqnarray}
\begin{rem}\label{022}
Note that if $j,a,k,k'$ are specified, and then $A,K,K'$ are uniquely determined.
\end{rem}
In view of (\ref{018}) and (\ref{021}), we have
\begin{eqnarray*}
I&\leq&\frac{1}{2}\sum_{j}\sum_{*}\sum_{**}(k_j+k_j')(K_j+K_j')
e^{-\frac{(2-2^{\theta})\delta_{1}}3\sum_{i\geq 1}(n_i^*)^{\theta}}e^{-(2-2^{\theta})\delta_{2}\sum_{i\geq 3}(N_i^*)^{\theta}}\\
&&\times e^{-\left(\delta_{1}m^*(k,k')^{\theta}+\delta_{2}m^*(K,K')^{\theta}\right)}\\
&\leq&\frac{1}{2}\sum_{a,k,k'}\sum_{j}(k_j+k_j')(K_j+K_j')
e^{-\frac{(2-2^{\theta})\delta_{1}}3\sum_{i\geq 1}(n_i^*)^{\theta}}e^{-(2-2^{\theta})\delta_{2}\sum_{i\geq 3}(N_i^*)^{\theta}}\\
&&\mbox{(in view of Remark \ref{022}, one has $\sum_{a,k,k'}\sum_{j}=\sum_{j}\sum_{*}\sum_{**}$)}\\
&\leq&\sum_{a,k,k'}\left(\sum_{i\geq 1}(n_i^*)^{\theta}\right)\left(\sum_{i\geq 3}(N_i^*)^{\theta}\right)e^{-\frac{(2-2^{\theta})\delta_{1}}3\sum_{i\geq 1}(n_i^*)^{\theta}}e^{-(2-2^{\theta})\delta_{2}\sum_{i\geq 3}(N_i^*)^{\theta}}\\
&&\mbox{(in view of the inequality (\ref{017}))}\\
&\leq&\sum_{a,k,k'}\left(\sum_{i\geq 1}(n_i^*)^{\theta}e^{-\frac{(2-2^{\theta})\delta_{1}}3\sum_{i\geq 1}(n_i^*)^{\theta}}\right)\left(\sum_{i\geq 3}(N_i^*)^{\theta}e^{-(2-2^{\theta})\delta_{2}\sum_{i\geq 3}(N_i^*)^{\theta}}\right)\\
\nonumber &\leq&\frac{12}{(2-2^{\theta})^2\delta_{1}\delta_{2}}
\sum_{a,k,k'}
e^{-\frac{(2-2^{\theta})\delta_{1}}4\sum_{i\geq 1}(n_i^*)^{\theta}}\qquad \mbox{(in view of Lemma 7.5 in \cite{CLSY})}\\
&=&\frac{12}{(2-2^{\theta})^2\delta_{1}\delta_{2}}\sum_{a,k,k'}
e^{-\frac{(2-2^{\theta})\delta_{1}}4\sum_{n}(2a_n+k_n+k'_n)|n|^{\theta}}\\
&\leq&\frac{12}{(2-2^{\theta})^2\delta_{1}\delta_{2}}\left(\sum_{a}
e^{-\frac{(2-2^{\theta})\delta_{1}}4\sum_{n}2a_n|n|^{\theta}}\right)
\left(\sum_{k}e^{-\frac{(2-2^{\theta})\delta_{1}}4\sum_{n}k_n|n|^{\theta}}\right)^2
\\
&\leq&\frac{12}{(2-2^{\theta})^2\delta_{1}\delta_{2}}\prod_{n\in\mathbb{Z}}
\left({1-e^{-\frac{(2-2^{\theta})\delta_1}2|n|^{\theta}}}\right)^{-1}
\left({1-e^{-\frac{(2-2^{\theta})\delta_{1}}4|n|^{\theta}}}\right)^{-2}      \\
%&&\nonumber\mbox{(which is based on Lemma 7.5 in \cite{cl2018th})}\\
&\leq&\frac{C_1(\theta)}{\delta_{1}\delta_{2}}\left(\frac{1}{\delta_{1}}\right)^{C_2{(\theta)}{\delta_{1}^{-\frac{1}{\theta}}}}\ \ \mbox{(in view of Lemma 7.4 in \cite{CLSY})}\\
&\leq&\frac{1}{\delta_{2}}\left(\frac{1}{\delta_{1}}\right)^{C{(\theta)}{\delta_{1}^{-\frac{1}{\theta}}}},
\end{eqnarray*}
where the last inequality is based on $0<\delta_{1},\delta_{2}\ll1$ and $C(\theta), C_1(\theta), C_2(\theta)$ are positive constants depending on $\theta$ only.

$\textbf{Case. 1.2.}\ j\in\{n_1,n_2\},\ |n_1|=n_1^*,\ |n_2|=n_2^*.$

If $2a_j+k_j+k'_j>2$, then $|j|\leq n_3^*$, we are in $\textbf{Case. 1.1.}$. Hence in what follows, we always assume
\begin{equation*}
2a_j+k_j+k'_j\leq2,
\end{equation*}
which implies
\begin{equation}\label{023}
k_j+k_j'\leq 2
\end{equation}
and
\begin{equation}\label{024}
n_2^*>n_3^*.
\end{equation}
From (\ref{023}) and in view of $j\in\{n_1,n_2\}$, it follows that
\begin{eqnarray*}
I\nonumber
\nonumber&\leq&\sum_{a,k,k'}(K_{n_1}+K'_{n_1}+K_{n_2}+K'_{n_2}) \\
\label{025*}&&\times e^{-(2-2^{\theta})\left(\delta_{1}\sum_{i\geq 3}(n_i^*)^{\theta}+\delta_{2}\sum_{i\geq 3}(N_i^*)^{\theta}\right)-\delta_{1}m^*(k,k')^{\theta}}.
\end{eqnarray*}
Since
\begin{equation*}
 K_j+K'_j\leq \kappa_j+\kappa'_j-k_j-k'_j+2\leq\kappa_j+\kappa'_j+2, \forall j,
\end{equation*}
one has
\begin{eqnarray}
\nonumber I\nonumber&\leq& \sum_{a,k,k'}(\kappa_{n_1}+\kappa'_{n_1}+\kappa_{n_2}+\kappa'_{n_2}+4) \\
\nonumber&&\times e^{-(2-2^{\theta})\left(\delta_{1}\sum_{i\geq 3}(n_i^*)^{\theta}+\delta_{2}\sum_{i\geq 3}(N_i^*)^{\theta}\right)-\delta_{1}m^*(k,k')^{\theta}}\\
 &\leq& \nonumber \sum_{a,k,k'}(\kappa_{n_1}+\kappa'_{n_1}+\kappa_{n_2}+\kappa'_{n_2}+4) \\
\nonumber&&\times e^{-\frac12(2-2^{\theta})\delta_{1}\sum_{i\geq 3}(n_i^*)^{\theta}-\delta_{1} m^*(k,k')^{\theta}}\  \mbox{ ( based  on  (\ref{009*}) )}\\
&&\nonumber \times e^{-\frac14(2-2^\theta)\delta_{1}\wedge\delta_{2}\sum_{n}(2\alpha_n+\kappa_n+\kappa'_n)}\\
\label{025} &=& \sum_{l\in\mathbb{Z}}\sum_{a,k,k',\atop{m(k,k')=l}}(\kappa_{n_1}+\kappa'_{n_1}+\kappa_{n_2}+\kappa'_{n_2}+4) \\
\nonumber&&\times e^{-\frac12(2-2^{\theta})\delta_{1}\sum_{i\geq 3}(n_i^*)^{\theta}-\delta_{1} |l|^{\theta} }\\
&&\nonumber \times e^{-\frac14(2-2^\theta)\delta_{1}\wedge\delta_{2}\sum_{n}(2\alpha_n+\kappa_n+\kappa'_n)},
\end{eqnarray}
where $\delta_{1}\wedge\delta_{2}=\min\{\delta_{1},\delta_{2}\} $.
\begin{rem}\label{026} Obviously, $\{n_1,n_2\}\bigcap \mathrm{supp}\ \mathcal{M}_{\alpha\kappa\kappa'}\neq \emptyset$, and if $n_1$ (resp. $n_2$), $\{n_i\}_{i\geq 3}$ and $m(k,k')=l$  is specified, then $n_2$ (resp. $n_1$) is determined uniquely. Thus $n_1,n_2$ range in a set of cardinality no more than \begin{equation}\label{027}\#\mathrm{supp} \ \mathcal{M}_{\alpha\kappa\kappa'}\leq\sum_{n}(2\alpha_n+\kappa_n+\kappa'_n).
\end{equation}
\end{rem}
Also, if $\{n_i\}_{i\geq 1}$ is given, then $\{2a_n+k_n+k'_n\}_{n\in\mathbb{Z}}$ is specified, and hence $(a,k,k')$ is specified up to a factor of
$$\prod_{n}(1+l_n^2),$$
where
$$l_n=\#\{j:n_j=n\}.$$
Following the inequality (\ref{025}), we thus obtain
\begin{eqnarray}
\nonumber I&\leq& \sum_{l\in\mathbb{Z}}\sum_{\{n_i\}_{i\geq1}}\prod_{m}(1+l_m^2)(\kappa_{n_1}+\kappa'_{n_1}+\kappa_{n_2}+\kappa'_{n_2}+4) \\
\nonumber&&\times e^{-\frac12(2-2^{\theta})\delta_{1}\sum_{i\geq 3}(n_i^*)^{\theta}-\delta_{1} |l|^{\theta}}\\
&&\nonumber\times e^{-\frac14(2-2^\theta)\delta_{1}\wedge\delta_{2}\sum_{n}(2\alpha_n+\kappa_n+\kappa'_n)}\\
\nonumber&\leq& 5\sum_{l\in\mathbb{Z}}\sum_{\{n_i\}_{i\geq3}}\prod_{|m|\leq n_3^*}(1+l_m^2)\left(\sum_{n_1,n_2}(\kappa_{n_1}+\kappa'_{n_1}+\kappa_{n_2}+\kappa'_{n_2}+4)\right) \\
\nonumber&& \times e^{-\frac12(2-2^{\theta})\delta_{1}\sum_{i\geq 3}(n_i^*)^{\theta}
-\delta_{1} |l|^{\theta}}\\
&&\nonumber\times e^{-\frac14(2-2^\theta)\delta_{1}\wedge\delta_{2}\sum_{n}(2\alpha_n+\kappa_n+\kappa'_n)}\\
&&\nonumber\mbox{(in view of $\prod_{|m|>n^*_1}(1+l_m^2)=1$ and $\prod_{m\in\{n_1,n_2\}}(1+l_m^2)\leq 5$)}\\
\nonumber&\leq& 5\sum_{l\in\mathbb{Z}}\sum_{\{n_i\}_{i\geq3}}\prod_{|m|\leq n_3^*}(1+l_m^2)\left(\sum_n(\kappa_n+\kappa_n')+\sum_{n_1,n_2}4\right) \\
\nonumber&& \times e^{-\frac12(2-2^{\theta})\delta_{1}\sum_{i\geq 3}(n_i^*)^{\theta}-\delta_{1} |l|^{\theta}}\\
&&\nonumber\times e^{-\frac14(2-2^\theta)\delta_{1}\wedge\delta_{2}\sum_{n}(2\alpha_n+\kappa_n+\kappa'_n)}\\
\nonumber& \leq & 5\sum_{l\in\mathbb{Z}}\sum_{\{n_i\}_{i\geq3}}\prod_{|m|\leq n_3^*}(1+l_m^2)\left(\sum_{n}(\kappa_{n}+\kappa'_{n})+4\#\mathrm{supp}\ \mathcal{M}_{\alpha\kappa\kappa'}\right) \\
\nonumber&& \times e^{-\frac12(2-2^{\theta})\delta_{1}\sum_{i\geq 3}(n_i^*)^{\theta}-\delta_{1} |l|^{\theta}}\\
&&\nonumber\times e^{-\frac14(2-2^\theta)\delta_{1}\wedge\delta_{2}\sum_{n}(2\alpha_n+\kappa_n+\kappa'_n)}\\
&&\nonumber\mbox{(the inequality is based on Remark \ref{026})}\\
\nonumber&\leq& \frac{C_3({\theta})}{\delta_{1}\wedge\delta_{2}}\left(\sum_{l\in\mathbb{Z}}e^{-\delta_{1} |l|^{\theta}}\right)\left(\sum_{\{n_i\}_{i\geq3}}\prod_{|m|\leq n_3^*}(1+l_m^2)e^{-\frac12(2-2^{\theta})\delta_{1}\sum_{i\geq 3}(n_i^*)^{\theta}}\right)\\
&&\mbox{ (based on (\ref{027}) and Lemma 7.5 in \cite{CLSY})}\nonumber\\
\nonumber&\leq&\frac{C_3({\theta})}{\delta_{1}\wedge\delta_{2}}\left(\sum_{l\in\mathbb{Z}}e^{-\delta_{1} |l|^{\theta}}\right)\left(\sum_{\{l_m\}_{|m|\leq n_3^{\ast}}}
e^{-\frac13(2-2^{\theta})\delta_{1}\sum_{|m|\leq n_3^\ast}l_m|m|^{\theta}}\right)
\\
\nonumber&&\times\sup_{\{l_m\}_{|m|\leq n_3^\ast}}\left(\prod_{|m|\leq n_3^\ast}(1+l_m^2)
e^{-\frac16(2-2^{\theta})\delta_{1}\sum_{|m|\leq n_3^\ast}l_m|m|^{\theta}}\right)\\
\nonumber&\leq&\frac{C_3({\theta})}{\delta_{1}\wedge\delta_{2}}
\left(\frac{1}{\delta_{1}}\right)^{C_4({\theta}){\delta_{1}^{-\frac{1}{\theta}}}}
\left(\sum_{l\in\mathbb{Z}}e^{-\delta_{1} |l|^{\theta}}\right)\left(\sum_{\{l_m\}_{|m|\leq n_3^\ast}}
e^{-\frac13(2-2^{\theta})\delta_{1}\sum_{|m|\leq n_3^\ast}l_m|m|^{\theta}}\right)\\
\nonumber&&\mbox{(in view of Lemma 7.7 in \cite{CLSY})}\\
\nonumber& \leq & \frac{C_3({\theta})}{\delta_{1}\wedge\delta_{2}}\left(\frac{1}{\delta_{1}}\right)^{C_4({\theta}){\delta_{1}^{-\frac{1}{\theta}}}}\prod_{m\in\mathbb{Z}}\frac{1}{1-e^{-\frac13(2-2^{\theta})\delta_{1} |m|^{\theta}}}\left(\frac{1}{\delta_{1}}\right)^{\frac{1}{\theta}}\quad\\
&&\nonumber \mbox{(in view of Lemma 7.2 and Lemma 7.3 in \cite{CLSY})}\\
\nonumber&\leq&\frac{1}{\delta_{2}}\left(\frac{1}{\delta_{1}}\right)^{C_{5}({\theta}){\delta_{1}^{-\frac{1}{\theta}}}},
\end{eqnarray}
where the last equality is based on Lemma 7.4 in \cite{CLSY} and $C_{3}(\theta), C_{4}(\theta), C_{5}(\theta)$ are positive constants depending on $\theta$ only.

$\textbf{Case. 2.}\ \nu_1^*>N_1^*.$

In view of (\ref{014}), one has $n_1^*=\nu_1^*$. Hence,  $n_2$ is determined by $n_1$, $\{n_i\}_{i\geq 3}$ and the momentum $m(k,k')$. Similar to Case 1.2, we have
\begin{eqnarray*}
I&\leq&\frac{1}{\delta_{2}}\left(\frac{1}{\delta_{1}}\right)^{C_{6}({\theta}){\delta_{1}^{-\frac{1}{\theta}}}},
\end{eqnarray*}
where $C_{6}(\theta)$ is some positive constant depending on $\theta$ only.

Therefore,\ we finish the proof of (\ref{H4}).
\end{proof}
\subsection{Proof of Lemma \ref{H6}}
\begin{proof}
In view of (\ref{028}) and for each $ j \in \mathbb{Z}$,\ one has
\begin{equation*}
\frac{\partial{H}}{\partial q_{j}}=\sum_{a,k,k'}B_{akk'}\left(\prod_{n\neq j}I_{n}(0)^{a_{n}}q_{n}^{k_{n}}
\bar{q}_{n}^{k_{n}'}\right)\left(k_{j}
I_{j}(0)^{a_{j}}q_{j}^{k_{j}-1}\bar{q}_{j}^{k_{j}'}\right).
\end{equation*}
Now we would like to estimate
\begin{equation}\label{031}
\left|e^{r|j|^{\theta}}\frac{\partial{H}}{\partial q_{j}}\right|=\left|e^{r|j|^{\theta}}\sum_{a,k,k'}B_{akk'}\left(\prod_{n\neq j}I_{n}(0)^{a_{n}}q_{n}^{k_{n}}
\bar{q}_{n}^{k_{n}'}\right)\left(k_{j}
I_{j}(0)^{a_{j}}q_{j}^{k_{j}-1}
\bar{q}_{j}^{k_{j}'}\right)\right|.
\end{equation}
Based on (\ref{H00}), one has
\begin{equation}\label{032}
|B_{akk'}|\leq ||H||_{\rho,\mu}e^{\rho(\sum_{n}(2a_{n}+k_{n}+k_{n}')|n|^{\theta}
-2(n_1^\ast)^{\theta})-\mu m^\ast(k,k')^{\theta}}.
\end{equation}
In view of $||q||_{r,\infty}<1$ and $ ||I(0)||_{r,\infty} < 1$, one has
\begin{eqnarray}\label{033}
|q_{n}| < e^{-r|n|^{\theta}},
\end{eqnarray}
and
\begin{eqnarray}\label{034}
|I_{n}(0)| < e^{-2r|n|^{\theta}}.
\end{eqnarray}
Substituting (\ref{033}) and (\ref{034}) into (\ref{032}),\ one has
\begin{eqnarray}
&&|(\ref{032})|\nonumber\\
&\leq&\nonumber||H||_{\rho,\mu}\left|e^{r|j|^{\theta}}\sum_{l\in\mathbb{Z}}\sum_{a,k,k',\atop{m(k,k')=l}}
k_je^{\rho(\underset{n}{\sum}(2a_n+k_n+k_n')|n|^{\theta}-2(n_1^*)^{\theta})}
e^{-r(\underset{n}{\sum}(2a_n+k_n+k_n')|n|^{\theta}-|j|^{\theta})-\mu|l|^{\theta}}\right|\\
&=&\nonumber||H||_{\rho,\mu}\left|\sum_{l\in\mathbb{Z}}\sum_{a,k,k',\atop{m(k,k')=l}}
k_je^{\rho(\underset{n}{\sum}(2a_n+k_n+k_n')|n|^{\theta}-2(n_1^*)^{\theta})}
e^{-r(\underset{n}{\sum}(2a_n+k_n+k_n')|n|^{\theta}-2|j|^{\theta})-\mu|l|^{\theta}}\right|.
\end{eqnarray}
Now we will estimate the last inequality in the following two cases:

\textbf{Case 1.} $|j|\leq n_3^*$.

Then one has
\begin{eqnarray*}
&&\left|\sum_{l\in\mathbb{Z}}\sum_{a,k,k',\atop{m(k,k')=l}}k_j
e^{\rho(\underset{n}{\sum}(2a_n+k_n+k_n')|n|^{\theta}-2(n_1^*)^{\theta})}
e^{-r(\underset{n}{\sum}(2a_n+k_n+k_n')|n|^{\theta}-2|j|^{\theta})-\mu|l|^{\theta}}\right|\\
&\leq& \left|\sum_{a,k,k'}k_je^{\rho\sum_{i\geq 1}(n_i^*)^{\theta}}e^{-r(n_1^*)^{\theta}-r\sum_{i\geq 4}(n_i^*)^{\theta}}\right|\left( \sum_{l\in \mathbb{Z}}e^{-\mu|l|^{\theta}}\right)\quad \mbox{(in view of $|j|\leq n_3^*$)}\\
\nonumber&\leq&\sum_{a,k,k'}\left(\sum_{i\geq 1}(n_i^*)^{\theta}\right)e^{\frac13(-r+3\rho)\sum_{i\geq 1}(n_i^*)^{\theta}}\left( \sum_{l\in \mathbb{Z}}e^{-\mu|l|^{\theta}}\right)\\
\nonumber&\leq& \left( \frac{1}{\mu}\right)^{\frac{1}{\theta}}\left(\frac{12}{r-3\rho}\right)\left(\sum_{a,k,k'}
e^{\frac14(-r+3\rho)\sum_{i\geq 1}(n_i^*)^{\theta}}\right)\\
&&\nonumber\quad\mbox{(in view of  Lemma 7.5 in \cite{CLSY})}\\
\nonumber&=& \left( \frac{1}{\mu}\right)^{\frac{1}{\theta}}\left(\frac{12}{r-3\rho}\right) \sum_{a,k,k'}
e^{\frac14(-r+3\rho)\underset{n}{\sum}(2a_n+k_n+k'_n)|n|^{\theta}}\\
\nonumber&\leq&\left( \frac{1}{\mu}\right)^{\frac{1}{\theta}}\left(\frac{12}{r-3\rho}\right)\left(\sum_{a}
e^{\frac14(-r+3\rho)\sum_{n}2a_n|n|^{\theta}}\right)\times\left(\sum_{k}e^{\frac14(-r+3\rho)\underset{n}{\sum}k_n|n|^{\theta}}\right)^2\\
\nonumber&\leq&\left( \frac{1}{\mu}\right)^{\frac{1}{\theta}}\left(\frac{12}{r-3\rho}\right)\prod_{n\in\mathbb{Z}}
\left({1-e^{\frac12(-r+3\rho)|n|^{\theta}}}\right)^{-1}\times\prod_{n\in\mathbb{Z}}
\left({1-e^{\frac14(-r+3\rho)|n|^{\theta}}}\right)^{-2}\\
&&\nonumber\mbox{(in view of Lemma 7.2 in \cite{CLSY})}\\
\nonumber&\leq& \left( \frac{1}{\mu}\right)^{\frac{1}{\theta}}\left(\frac{12(2-2^{\theta})}{\rho}\right)\prod_{n\in\mathbb{Z}}
\left({1-e^{-\frac{1}{2}\rho|n|^{\theta}}}\right)^{-1}
\nonumber\prod_{n\in\mathbb{Z}}
\left({1-e^{-\frac{1}{4}\rho|n|^{\theta}}}\right)^{-2}\\
&&\nonumber  \ {\mbox{(in view of $r>(\frac{1}{2-2^{\theta}}+3)\rho$)}}    \\
&\leq&\nonumber  \left( \frac{1}{\mu}\right)^{\frac{1}{\theta}} \left(\frac{12(2-2^{\theta})}{\rho}\right)\left(\frac{1}{\rho}\right)^{C_1{(\theta)}{\rho^{-\frac{1}{\theta}}}},
\end{eqnarray*}
where the last inequality is based on Lemma 7.4 in \cite{CLSY} and $C_1(\theta)$ is a positive constant depending on $\theta$ only.

\textbf{Case 2.} $|j|> n_3^*$, which implies $k_j\leq 2$.

Then one has
\begin{eqnarray}
\nonumber&&\left|\sum_{l\in\mathbb{Z}}\sum_{a,k,k',\atop{m(k,k')=l}}
k_je^{\rho(\underset{n}{\sum}(2a_n+k_n+k_n')|n|^{\theta}-2(n_1^*)^{\theta})}
e^{-r(\underset{n}{\sum}(2a_n+k_n+k_n')|n|^{\theta}-2|j|^{\theta})}e^{-\mu|l|^{\theta}}\right|\\
\nonumber&\leq&2 \left|\sum_{l\in\mathbb{Z}}\sum_{a,k,k',\atop{m(k,k')=l}}e^{\rho\sum_{i\geq 3}(n_i^*)^{\theta}}e^{-(2-2^{\theta})r\sum_{i\geq 3}(n_i^*)^{\theta}}e^{-(\mu-r)|l|^{\theta}}\right|\quad{(\mbox{in view of (\ref{H2}) and $k_j\leq2$})}\\
\nonumber&=& 2\left|\sum_{l\in\mathbb{Z}}\sum_{a,k,k',\atop{m(k,k')=l}}e^{(-(2-2^{\theta})r+\rho)\sum_{i\geq 3}(n_i^*)^{\theta}}e^{-(\mu-r)|l|^{\theta}}\right|\\
\nonumber &:=& A.
\end{eqnarray}
If $\{n_i\}_{i\geq 1}$ is given, then $\{2a_n+k_n+k'_n\}_{n\in\mathbb{Z}}$ is specified, and hence $(a,k,k')$ is specified up to a factor of
$$\prod_{n}(1+l_n^2),$$
where
$$l_n=\#\{j:n_j=n\}.$$
Since $|j|>n_3^*$, then $j\in\{n_1,n_2\}$. Hence, if $(n_i)_{i\geq 3}$ and $j,m^{\ast}(k,k')$ are given, then $n_1$ and $n_2$ are uniquely determined. Then, one has
\begin{eqnarray}
\nonumber A &\leq&  2\left|\sum_{l\in\mathbb{Z}}\sum_{(n_i)_{i\geq3}}\prod_{|n|\leq n_1^*}(1+l_n^2)e^{(-(2-2^{\theta})r+\rho)\sum_{i\geq 3}(n_i^*)^{\theta}}e^{-(\mu-r)|l|^{\theta}}\right| \\
\nonumber&\leq&10\left|\sum_{l\in\mathbb{Z}}\sum_{(n_i)_{i\geq3}}\prod_{|n|\leq n_3^*}(1+l_n^2)e^{(-(2-2^{\theta})r+\rho)\sum_{i\geq 3}(n_i^*)^{\theta}}e^{-(\mu-r)|l|^{\theta}}\right| \\
\nonumber && \mbox{ ( in view of $ \prod_{n \in \{n_1, n_{2}\}}(1+l_n^2) \leq 5 )$ }\\
\nonumber&\leq&10\left(\sum_{(n^*_i)_{i\geq3}}e^{-2(2-2^{\theta})\rho\sum_{i\geq 3}(n_i^*)^{\theta}}\right)\sup_{(n^\ast_i)_{i\geq3}}\left(\prod_{|n|\leq n_3^*}(1+l_n^2)e^{-(2-2^{\theta})\rho\sum_{i\geq 3}(n_i^*)^{\theta}}\right)\\
\nonumber && \times\left(\sum_{l\in\mathbb{Z}}e^{-(\mu-r)|l|^{\theta}}\right)\  \mbox{ ( in view of $ r > (\frac{1}{2-2^{\theta}} +3)\rho $ ) }\\
\nonumber&\leq&10\left(\frac{1}{\mu-r}\right)^{\frac{1}{\theta}}\left(\frac{1}{\rho}\right)^{C_2{(\theta)}
{\rho^{-\frac{1}{\theta}}}}\left(\sum_{(n^*_i)_{i\geq3}}e^{-2(2-2^{\theta})\rho\sum_{i\geq 3}(n_i^*)^{\theta}}\right)\\
\nonumber &&\mbox{( in view of Lemma 7.3 and Lemma 7.7 in \cite{CLSY})}\\
\nonumber&=&10\left(\frac{1}{\mu-r}\right)^{\frac{1}{\theta}}\left(\frac{1}{\rho}\right)^{C_2{(\theta)}
{\rho^{-\frac{1}{\theta}}}}\left(\sum_{(l_n)_{|n|\leq n_3^*}}e^{-2(2-2^{\theta})\rho\sum_{|n|\leq n_3^*}l_n|n|^{\theta}}\right)\\
\nonumber&\leq&C_1(\theta)\left(\frac{1}{\mu-r}\right)^{\frac{1}{\theta}}
\left(\frac{1}{\rho}\right)^{C_3{(\theta)}{\rho^{-\frac{1}{\theta}}}},
\end{eqnarray}
where the last equality is based on Lemma 7.2 and Lemma 7.4 in \cite{CLSY},\ and $ C_1(\theta), C_2(\theta), C_{3}(\theta)$ are positive constants depending on $\theta$ only.\\
Hence, we finished the proof of (\ref{029}).
\end{proof}

\bibliographystyle{abbrv} % abbrv
%\bibliography{references}

% ------------------------------------------------------------------------
\end{document}